\documentclass[11pt,a4paper]{article}

\usepackage{amsmath,amssymb,amsthm}

\allowdisplaybreaks
\numberwithin{equation}{section}

\theoremstyle{plain}
\newtheorem{theorem}{Theorem}[section]
\newtheorem{proposition}[theorem]{Proposition}
\newtheorem{lemma}[theorem]{Lemma}
\newtheorem{corollary}[theorem]{Corollary}
\theoremstyle{definition}

\theoremstyle{remark}

\newcommand{\R}{\mathbb R}
\newcommand{\E}{\mathbb E}
\newcommand{\Pp}{\mathbb P}
\newcommand{\dd}{\,\mathrm d}
\newcommand{\Df}{\mathcal D_f}
\newcommand{\Ff}{\mathsf F_f}
\newcommand{\Astar}{\mathop{\star}}
\newcommand{\one}{\mathbf 1}
\newcommand{\Law}{\operatorname{Law}}
\newcommand{\Pois}{\operatorname{Pois}}
\newcommand{\Ber}{\operatorname{Bernoulli}}
\newcommand{\cx}{\preceq_{\mathrm{cx}}}
\newcommand{\pos}[1]{\left[#1\right]_+}
\newcommand{\Aff}{\operatorname{Aff}}
\newcommand{\TV}{\operatorname{TV}}
\newcommand{\persp}{\overline f}
\newcommand{\Cuniv}{\mathcal C}

\title{A local-to-global principle for\\
$f$-divergence homogenization}
\author{Aryeh Kontorovich\\[4pt]
\small Institute for the Foundations of AI\\
\small Ben-Gurion University of the Negev\\
\small Beer-Sheva, Israel\\[3pt]
\small\texttt{karyeh@bgu.ac.il}}
\date{}

\begin{document}
\maketitle

\begin{abstract}
We study the effect of replacing a family of independent
statistical experiments by independent repetitions of their average.
We give a simple, complete characterization of the $f$-divergences for
which this operation increases divergence by at most a constant
depending only on the generator, uniformly in the number of
observations and the underlying measurable space. The characterization
is governed by a local-to-global principle: a comparison restricted to
two observations on a binary alphabet already determines the validity
of the general inequality. It admits an explicit scalar formulation,
which simplifies further for symmetric divergences. The result applies
to arbitrary convex generators and requires no differentiability or
symmetry assumptions. It recovers the homogenization inequality for
total variation and identifies exactly the admissible members of the
power-divergence family. In particular, whenever uniform homogenization
fails, the obstruction is already present in a family of binary
experiments with two observations. Moreover, the global comparison
constant is controlled uniformly over generators with a common bound
in the local inequality.
\end{abstract}

\medskip
\noindent\textbf{Keywords:} $f$-divergence, product measures, homogenization,
convex order, moment inequalities, statistical experiments.

\smallskip
\noindent\textbf{MSC Classification:} 60E15, 62B10, 94A17.

\section{Introduction}
\label{sec:introduction}

Let $P_1,\ldots,P_n$ and $Q_1,\ldots,Q_n$ be probability measures on a
common measurable space, and set
\[
 \bar P=\frac1n\sum_{i=1}^nP_i,
 \qquad
 \bar Q=\frac1n\sum_{i=1}^nQ_i.
\]
We characterize the $f$-divergences for which
\begin{equation}
 D_f(\bar P^{\otimes n}\|\bar Q^{\otimes n})
 \le C_f\,
 D_f\!\left(\bigotimes_{i=1}^nP_i\middle\|
                  \bigotimes_{i=1}^nQ_i\right)
 \label{eq:homogenization}
\end{equation}
holds with a finite constant $C_f$ depending only on $f$, uniformly in
$n$, the measures, and the underlying space. We refer to
\eqref{eq:homogenization} as the \emph{homogenization inequality}.

Throughout, $f:[0,\infty)\to(-\infty,\infty]$ is lower
semicontinuous and convex, finite on $(0,\infty)$, and satisfies
$f(1)=0$. For $P\ll Q$,
\[
 D_f(P\|Q)=\E_Q f\!\left(\frac{\dd P}{\dd Q}\right);
\]
the extension to arbitrary pairs is specified in
Section~\ref{sec:likelihoods}. This is the usual class of divergences
introduced in \cite{AliSilvey1966,Csiszar1967}.

The two sides of \eqref{eq:homogenization} correspond to distinct
sampling designs for a binary statistical experiment. The product on
the right uses each coordinate once; the product on the left samples
coordinate indices independently and uniformly, with the indices
unobserved. For total variation, the identity
\begin{equation}
 e(P,Q)=\frac{1-\TV(P,Q)}2
 \label{eq:testing-error}
\end{equation}
for the equal-prior Bayes error identifies \eqref{eq:homogenization}
with a comparison of testing advantages. More generally,
$f$-divergences admit integral representations in terms of the
reductions in Bayes risk for binary experiments as the prior varies
\cite[Sections~4--5]{ReidWilliamson2011}.

Joint convexity gives the one-coordinate bound
\[
 D_f(\bar P\|\bar Q)
 \le\frac1n\sum_{i=1}^nD_f(P_i\|Q_i),
\]
but does not by itself control its interaction with tensorization.
Product additivity yields \eqref{eq:homogenization} with $C_f=1$ for
relative entropy and its reverse, and multiplicativity of affinity
gives the same conclusion for squared Hellinger distance. The
total-variation inequality was established for Bernoulli products in
\cite{Kontorovich2026Bernoulli} and for arbitrary measures in
\cite{Kontorovich2026TV}. Pearson's $\chi^2$ divergence fails
\eqref{eq:homogenization} even for $n=2$ on a binary alphabet.

The main result identifies this last observation as a general
local-to-global principle. For every admissible generator,
homogenization in the two-observation binary case is equivalent to
homogenization for arbitrary product experiments. Equivalently, it
suffices to control the divergence of the product of two independent
copies of an arbitrary mean-one two-point likelihood ratio. This
criterion yields a complete scalar characterization of the generators.

The proof combines an ordered two-point decomposition, comparison
inequalities for independent sums, and Poissonization. The convex
thinning estimate in Theorem~\ref{thm:centered-erasure} provides the
uniformity required near the diagonal. We discuss the relation to
earlier work in Section~\ref{sec:related} and outline the argument in
Section~\ref{sec:strategy}.

\section{A local-to-global characterization}
\label{sec:main}

Replacing $f(t)$ by $f(t)-a(t-1)$, where $a\in\partial f(1)$,
leaves the divergence unchanged. We henceforth use this normalization,
so that
\begin{equation}
 f\ge0,\qquad f(1)=0.
 \label{eq:nonnegative-normalization}
\end{equation}
The identically zero case is immediate.

\begin{theorem}[Local-to-global principle]
\label{thm:main}
For a generator $f$ as above, the following are equivalent.
\begin{enumerate}
 \item The homogenization inequality \eqref{eq:homogenization} holds
 with a finite constant $C_f$.
 \item There is a finite $K$ such that
 \begin{equation}
 D_f\!\left(\left(\frac{P_1+P_2}{2}\right)^{\otimes2}
       \middle\|
       \left(\frac{Q_1+Q_2}{2}\right)^{\otimes2}\right)
 \le K D_f(P_1\otimes P_2\|Q_1\otimes Q_2)
 \label{eq:Ber2}
 \end{equation}
 for all $P_1,P_2,Q_1,Q_2$ on $\{0,1\}$.
 \item There is a finite $A$ such that, whenever $X$ is positive,
 takes at most two values, and satisfies $\E X=1$, an independent
 copy $X'$ satisfies
 \begin{equation}
 \E f(XX')\le A\,\E f(X).
 \label{eq:two-point-test}
 \end{equation}
\end{enumerate}
Moreover, there is a universal nondecreasing function $\Cuniv$ such
that a constant $K$ in \textnormal{(ii)} gives
$C_f\le\Cuniv(16K+4)$ in \textnormal{(i)}. A constant $A$ in
\textnormal{(iii)} gives $K\le A$ and $C_f\le\Cuniv(2A)$.
\end{theorem}

Both the sample-size and alphabet restrictions in
\textnormal{(ii)} are fixed independently of $f$. Thus failure of
uniform homogenization is always witnessed within the two-observation
binary class. The modulus $\Cuniv$ is uniform over generators with a
common local bound; no sharp dependence is asserted.

\subsection{The scalar criterion}

For a mean-one variable with support $\{r,t\}$, where $0<r<1<t$,
\begin{equation}
 \Pp(X=r)=\frac{t-1}{t-r},
 \qquad
 \Pp(X=t)=\frac{1-r}{t-r}.
 \label{eq:intro-twopoint-probabilities}
\end{equation}
Expanding \eqref{eq:two-point-test} and bounding the mixed term by
convexity gives the following equivalent condition.

\begin{corollary}[Scalar characterization]
\label{cor:scalar-main}
The conditions of Theorem~\ref{thm:main} hold if and only if there is
a finite $M$ such that, for every $0<r<1<t$,
\begin{equation}
 \begin{split}
 &(t-1)^2 f(r^2)+(1-r)^2 f(t^2)\\
 &\hspace{12mm}\le
 M(t-r)\bigl[(t-1)f(r)+(1-r)f(t)\bigr].
 \end{split}
 \label{eq:endpoint-test}
\end{equation}
If $M$ and $A$ denote the least constants in
\eqref{eq:endpoint-test} and \eqref{eq:two-point-test}, respectively,
then $M\le A\le M+2$. In particular,
$C_f\le\Cuniv(2M+4)$.
\end{corollary}

For a symmetric generator, $f(t)=tf(1/t)$, the criterion reduces to
\begin{equation}
 f(t^2)\le C(t+1)f(t),\qquad t>1.
 \label{eq:symmetric-main}
\end{equation}
In the asymmetric case, separate growth conditions at zero and infinity
do not suffice; Proposition~\ref{ex:local-flat-generator} gives a
counterexample with bounded behavior at zero and linear growth at
infinity.

The criterion recovers total-variation homogenization and shows that
the Cressie--Read family is admissible exactly for parameters in
$[0,1]$. Throughout that interval the optimal constant is $1$; outside
it the two-observation binary comparison fails. The details, including
an explicit $\chi^2$ counterexample, are in
Section~\ref{sec:examples}.

\subsection{A thinning inequality}

The following estimate is used to obtain uniformity at small Hellinger
distance. Its statement is independent of the divergence setting.

\begin{theorem}[Bernoulli thinning of centered sums]
\label{thm:centered-erasure}
Let $Z_1,\ldots,Z_n$ be independent centered real variables, and let
$B_1,\ldots,B_n$ be independent $\Ber(q)$ variables, independent of the
$Z_i$, where $q\in[0,1]$. If $\psi:\R\to[0,\infty)$ is finite and
convex with $\psi(0)=0$, then
\begin{equation}
 \E\psi\left(\sum_iB_iZ_i\right)
 \le
 40\sqrt q\,
 \E\psi\left(\sum_iZ_i\right).
 \label{eq:centered-erasure}
\end{equation}
\end{theorem}

The proof is given in Section~\ref{app:centered}. Expectations may
take the value $+\infty$.

\section{Related work}
\label{sec:related}

The direct antecedents are the total-variation homogenization
inequalities in \cite{Kontorovich2026Bernoulli,Kontorovich2026TV}.
The Bernoulli case was established in
\cite[Theorem~1.5]{Kontorovich2026Bernoulli}. Its proof proceeds
through two-sided estimates for total variation between
Poisson-binomial laws under coordinatewise domination, using
parameter interpolation and second-moment extraction. The resulting
homogenization inequality holds for arbitrary Bernoulli parameter
vectors. The extension to arbitrary probability measures on a common
measurable space was obtained in
\cite[Theorem~4]{Kontorovich2026TV}, through a square-root-likelihood
representation by convolutions of finite measures and a lifting
argument that retains the coordinate labels. The present paper
extends this representation to general convex generators and
characterizes exactly those for which uniform homogenization holds.
The local-to-global equivalence simultaneously reduces the sample
size to two and the observation alphabet to two symbols.

The likelihood-ratio formulation belongs to the classical theory of
$f$-divergences \cite{AliSilvey1966,Csiszar1967}.
Reid and Williamson \cite{ReidWilliamson2011} give a systematic
treatment of the integral and variational representations connecting
divergences, Bayes risks, and statistical information. In particular,
these representations express a divergence through binary testing
problems with varying priors. The comparison in
\eqref{eq:homogenization} concerns the additional effect of averaging
coordinate laws before taking independent products. Its validity
depends on the generator, even though every $f$-divergence satisfies
data processing for a fixed experiment.

Binary-alphabet reductions also arise in comparisons between
divergences. Harremo\"es and Vajda \cite{HarremoesVajda2011} show that
the joint range of two $f$-divergences evaluated on the same pair of
measures is the convex hull of the range obtained on binary
alphabets. In the present problem, multiplicative convolution
introduces cross terms between components of a likelihood-law
mixture. The nested two-point decomposition in
Lemma~\ref{lem:local-nested-decomposition} orders the components in
convex order and controls their mixed products. The passage from
two factors to an arbitrary number of factors then requires the
convolution comparison of Theorem~\ref{thm:analytic}.

Roos \cite{Roos2010} studies approximation of a convolution
$F_1*\cdots*F_n$ by $\bar F^{*n}$, where
$\bar F=n^{-1}\sum_iF_i$, on measurable Abelian groups. His
total-variation bounds depend on the heterogeneity of the factors
and on convolution smoothing. This is a related averaging operation,
but the quantities being compared are different: approximation
controls the distance from one convolution to its homogenized law,
whereas \eqref{eq:homogenization} compares divergences between two
experiments before and after homogenization. Our comparison is
multiplicative and imposes no closeness condition among the
coordinate laws.

Klass's estimates for functions of independent sums
\cite{Klass1981} provide the probabilistic input for the comparison
away from the diagonal. Applied to log-likelihood sums with their
exponential normalization, they yield a bound with an additive
remainder; see Section~\ref{app:separated}. The convex thinning
inequality of Theorem~\ref{thm:centered-erasure} and the subsequent
Poissonization argument supply the multiplicative control needed
at arbitrarily small Hellinger distance.

\section{Outline of the proof}
\label{sec:strategy}

\subsection{Two-point reduction and convex order}

The likelihood representation reduces the problem to positive mean-one
laws under multiplicative convolution. Retaining the coordinate index
identifies the averaged likelihood law with the arithmetic mean of the
individual laws; data processing then yields the unlabelled comparison.
The principal reduction at the level of a single law is an ordered
two-point decomposition.

Every finitely supported mean-one law admits a mixture representation
by two-point mean-one laws whose support intervals are nested. These
components are totally ordered in convex order. Consequently, the
mixed product of two components is dominated by the square of the
larger component. Integrating the two-point square bound over this
decomposition gives the corresponding bound for arbitrary mean-one
laws with a loss of at most a factor of two. The construction and the
passage from the binary-experiment condition to the two-point square
bound are proved in Section~\ref{sec:local}.

It remains to pass from a two-factor comparison to arbitrary products.
This is the content of the convolution theorem,
Theorem~\ref{thm:analytic}. Its proof uses two estimates according to
the total squared Hellinger distance.

\subsection{Hellinger coordinates and the separated regime}

For a positive mean-one likelihood ratio $L$, introduce the finite
measure $\eta$ on $\R$ by
\[
 \int\varphi(x)\,\eta(\dd x)
 =\E\bigl[\sqrt L\,\varphi(\tfrac12\log L)\bigr],
 \qquad h(x)=e^{-x}f(e^{2x}).
\]
This correspondence is affine, takes multiplicative convolution to
additive convolution, and satisfies $\E f(L)=\int h\,\dd\eta$.
The transformed integrand is convex: for smooth $f$,
\[
 h''(x)=e^{-x}\bigl[f(e^{2x})+4e^{4x}f''(e^{2x})\bigr]\ge0,
\]
and the general case follows by approximation. Write
\[
 d_i=1-\E\sqrt{L_i},\qquad \Delta=\sum_i d_i.
\]

When $\Delta$ is bounded below, the comparison follows from Klass's
estimates \cite{Klass1981} and the normalization
\begin{equation}
 \E_Pe^{-X}=1,
 \qquad X=\log(\dd P/\dd Q),
 \label{eq:intro-exponential}
\end{equation}
for equivalent $P,Q$. The marginal quantities in Klass's estimates
are preserved by averaging the summand laws. The exponential
normalization controls the negative part of the sum, leaving an
additive error on a fixed divergence scale. A binary contraction of
the product experiment absorbs this error uniformly under the local
hypothesis; see Section~\ref{app:separated}.

\subsection{Thinning and repeated sampling}

For $\Delta\le1/4$, Theorem~\ref{thm:centered-erasure} gives a
multiplicative estimate without an additive error. Its transfer to
Hellinger coordinates requires control of centering: completing each
$\eta_i$ to a probability measure by adding its missing mass at zero
produces means satisfying $\sum_i|m_i|\le\Delta$. The translation
estimate in Proposition~\ref{prop:endpoint-reserve} controls the
resulting shifts. This yields a bound proportional to $\sqrt q$ for
independent thinning of the likelihood factors.

Poissonization replaces the sampling multiplicities by independent
$\Pois(1)$ variables $N_i$. The $k$th layer is then an independent
thinning of the original family with retention probability
\[
 q_k=\Pp\{\Pois(1)\ge k\}\le\frac1{k!}.
\]
The layers are combined by iterating the two-factor comparison along
a binary tree with depths chosen according to their probabilities.
Kraft's inequality and the factorial decay of $q_k$ make the resulting
bound summable. In fact, any positive power in the thinning estimate
suffices. Section~\ref{app:poisson} gives this implication, and
Section~\ref{app:completion} completes the argument and removes the
finite-support and positivity restrictions.

\section{Likelihood-law reduction}
\label{sec:likelihoods}

\subsection{Divergence and data processing}

Let $\lambda$ dominate $P,Q$, with densities $p,q$,
and put $f^\infty=\lim_{t\to\infty}f(t)/t$. Define
\[
 \persp(u,v)=
 \begin{cases}
 v f(u/v),&v>0,\\
 u f^\infty,&v=0,
 \end{cases}
 \qquad \persp(0,0)=0.
\]
Then
\begin{equation}
 D_f(P\|Q)=\int\persp(p,q)\,\dd\lambda.
 \label{eq:experiment-fdiv}
\end{equation}
Positive homogeneity makes the definition independent of $\lambda$.
The divergence is invariant under $f(t)\mapsto f(t)-a(t-1)$.
All inequalities are understood in $[0,\infty]$.

We use the data-processing inequality
\cite[Theorem~7.4]{PolyanskiyWu2025}: for every Markov kernel $T$,
\[
 D_f(PT\|QT)\le D_f(P\|Q).
\]
We also use joint convexity of $D_f$
\cite[Theorem~7.5(b)]{PolyanskiyWu2025}. For affinity and total
variation, we write
\[
 \Aff(P,Q)=\int\sqrt{pq}\,\dd\lambda,
 \qquad
 \TV(P,Q)=\frac12\int|p-q|\,\dd\lambda,
\]
and $1-\Aff(P,Q)\le\TV(P,Q)$.

\begin{lemma}[Removal of absolute continuity]
\label{lem:regularization}
A homogenization inequality proved for equivalent pairs $P_i\sim Q_i$ extends to
arbitrary pairs with the same constant.
\end{lemma}

\begin{proof}
Put $R_i=(P_i+Q_i)/2$ and replace each pair by
\[
 P_i^\varepsilon=(1-\varepsilon)P_i+\varepsilon R_i,
 \qquad
 Q_i^\varepsilon=(1-\varepsilon)Q_i+\varepsilon R_i.
\]
The measures are equivalent and are obtained from the original pair
by the channel that replaces the observation by an independent draw
from $R_i$ with probability $\varepsilon$. Data processing bounds the
regularized product divergence by the original one. The averages and
their fixed tensor powers converge in total variation as
$\varepsilon\downarrow0$. Joint lower semicontinuity of $f$-divergence
\cite{CsiszarTopological1967} completes the proof.
\end{proof}

\subsection{Multiplicative convolution}

For a probability law $\mu$ on $[0,\infty)$, write
\[
 \Df(\mu)=\int f(t)\,\mu(\dd t).
\]
The law is \emph{mean one} if $\int t\,\mu(\dd t)=1$.
If $X\sim\mu$ and $Y\sim\nu$ are independent, let
$\mu\star\nu=\Law(XY)$ denote multiplicative convolution. Its
identity is $\delta_1$; $\mu^{\star m}$ denotes the law of a product
of $m$ independent variables with law $\mu$, with
$\mu^{\star0}=\delta_1$.
For finite positive measures the same operation is defined by
integrating over the map $(x,y)\mapsto xy$.

\begin{proposition}[Likelihood representation]
\label{thm:likelihood-laws}
If $P_i\ll Q_i$ and $\mu_i=\Law_{Q_i}(\dd P_i/\dd Q_i)$, then
each $\mu_i$ is mean one and
\begin{equation}
 D_f\!\left(\bigotimes_iP_i\middle\|\bigotimes_iQ_i\right)
 =\Df(\mu_1\star\cdots\star\mu_n).
 \label{eq:likelihood-product}
\end{equation}
Every mean-one law can occur in this way.
\end{proposition}

\begin{proof}
The identity follows by factorizing the product likelihood ratio.
Conversely, take $Q=\mu$ and $P(\dd t)=t\mu(\dd t)$.
\end{proof}

\begin{proposition}[Labeled experiments]
\label{thm:lifting}
For the same likelihood laws, put $\bar\mu=n^{-1}\sum_i\mu_i$.
Then
\begin{equation}
 D_f(\bar P^{\otimes n}\|\bar Q^{\otimes n})
 \le\Df(\bar\mu^{\star n}).
 \label{eq:lifting-inequality}
\end{equation}
\end{proposition}

\begin{proof}
On $[n]\times\Omega$, assign mass $P_i(\dd x)/n$ and
$Q_i(\dd x)/n$ to the slice with label $i$.
The likelihood law is $\bar\mu$, so the divergence of $n$ independent
draws is $\Df(\bar\mu^{\star n})$. Projection onto the observation
coordinates gives the averaged product measures; apply data processing.
\end{proof}

The main analytic step is the following implication from two factors
to an arbitrary number of factors.

\begin{theorem}[Convolution comparison]
\label{thm:analytic}
There is a universal nondecreasing right-continuous function
$\Cuniv:[0,\infty)\to[1,\infty)$ with the following property.
Suppose that, for every pair of mean-one laws $\alpha,\beta$,
\begin{equation}
 \Df\!\left(\left(\frac{\alpha+\beta}{2}\right)^{\star2}\right)
 \le K\Df(\alpha\star\beta).
 \tag{B}\label{eq:B}
\end{equation}
Then every mean-one family $\mu_1,\ldots,\mu_n$ satisfies
\begin{equation}
 \Df\!\left(\left(\frac1n\sum_i\mu_i\right)^{\star n}\right)
 \le\Cuniv(K)\Df(\mu_1\star\cdots\star\mu_n).
 \tag{H}\label{eq:H}
\end{equation}
The same function $\Cuniv$ works for every admissible generator $f$.
\end{theorem}

The proof occupies Sections~\ref{app:global}--\ref{app:completion}.
Section~\ref{sec:local} shows that the
two-point condition in Theorem~\ref{thm:main} implies \eqref{eq:B}
with constant $2A$, and that the Bernoulli condition implies it with
constant $16K+4$. Propositions~\ref{thm:likelihood-laws}
and~\ref{thm:lifting}, followed by
Lemma~\ref{lem:regularization}, then give
\eqref{eq:homogenization}.

We write $X\cx Y$ if $X,Y$ have the same finite mean
and $\E\varphi(X)\le\E\varphi(Y)$ for every convex $\varphi$.

\begin{lemma}[Convex order and independent sums]
\label{lem:cx-convolution}
If $X\cx Y$ and $Z$ is independent of each, then
$X+Z\cx Y+Z$ whenever the variables are integrable.
\end{lemma}

\begin{proof}
Apply convex order to the convex function
$x\mapsto\E\varphi(x+Z)$. Iteration gives the corresponding assertion
for independent sums.
\end{proof}

\section{Two-point reduction}
\label{sec:local}

For $r<1<t$, let
\begin{equation}
 \rho_{r,t}=a\delta_r+b\delta_t,
 \qquad
 a=\frac{t-1}{t-r},\qquad b=\frac{1-r}{t-r}.
 \label{eq:local-two-point-law}
\end{equation}
This is the unique mean-one law on $\{r,t\}$. Since $r\le rt\le t$,
convexity controls the middle atom of $\rho_{r,t}^{\star2}$.

\begin{lemma}[Two-point comparison]
\label{lem:local-two-point-comparison}
The Bernoulli inequality \eqref{eq:Ber2} with constant $K$ implies
\eqref{eq:endpoint-test} with $M=8K$. The square bound
\eqref{eq:two-point-test} with constant $A$ implies
\eqref{eq:endpoint-test} with $M=A$; conversely, the endpoint condition
with constant $M$ implies the square bound with $A=M+2$. Finally, the
square bound with constant $A$ implies the Bernoulli inequality with the
same constant.
\end{lemma}

\begin{proof}
First assume the Bernoulli inequality with constant $K$. Realize
\eqref{eq:local-two-point-law} as the likelihood law of
\[
 Q_1=a\delta_0+b\delta_1,
 \qquad P_1=ar\delta_0+bt\delta_1.
\]
Set $P_2=Q_2=\delta_0$. In the averaged experiment, symbol $1$ has
likelihood ratio $t$ and $\bar Q$-probability $b/2$. The contribution
$b^2f(t^2)/4$ of $(1,1)$ and nonnegativity of $f$ give
\[
 b^2f(t^2)\le4K\Df(\rho_{r,t}).
\]
Taking instead $P_2=Q_2=\delta_1$ gives the corresponding lower-endpoint
bound $a^2f(r^2)\le4K\Df(\rho_{r,t})$. Adding them and multiplying by
$(t-r)^2$ proves \eqref{eq:endpoint-test} with $M=8K$.

Next,
\begin{equation}
 \Df(\rho_{r,t}\star\rho_{r,t})
 =a^2f(r^2)+2abf(rt)+b^2f(t^2).
 \label{eq:local-square-expansion}
\end{equation}
Since $r\le rt\le t$, convexity and nonnegativity imply
\[
 2abf(rt)\le2ab[f(r)+f(t)]
 \le2[af(r)+bf(t)]=2\Df(\rho_{r,t}).
\]
Since \eqref{eq:endpoint-test} is
$a^2f(r^2)+b^2f(t^2)\le M\Df(\rho_{r,t})$, the square and endpoint
bounds imply one another with $M=A$ and $A=M+2$, respectively.

Finally, suppose the square bound holds with constant $A$. For equivalent
Bernoulli experiments, the likelihood law $\rho$ of the averaged pair
$(\bar P,\bar Q)$ is itself a positive two-point mean-one law, or
$\delta_1$. Therefore
\begin{align*}
 D_f(\bar P^{\otimes2}\|\bar Q^{\otimes2})
 &\le A D_f(\bar P\|\bar Q)\\
 &\le\frac A2\bigl[D_f(P_1\|Q_1)+D_f(P_2\|Q_2)\bigr]\\
 &\le A D_f(P_1\otimes P_2\|Q_1\otimes Q_2).
\end{align*}
Here we used joint convexity and data processing.
Lemma~\ref{lem:regularization} removes the equivalence assumption.
\end{proof}

\subsection{Nested two-point decomposition}

To control cross terms in a mixture of two-point laws, we choose the
components to be ordered by convex order.

\begin{lemma}[Nested decomposition]
\label{lem:local-nested-decomposition}
Every finitely supported mean-one law $\mu$ on $[0,\infty)$ is a mixture
\[
 \mu=\sum_{j=0}^m\lambda_j\rho_j,
 \qquad \lambda_j\ge0,\qquad\sum_j\lambda_j=1,
\]
where $\rho_0=\delta_1$, each other $\rho_j$ is a two-point mean-one
law, and
\[
 \rho_0\cx\rho_1\cx\cdots\cx\rho_m.
\]
\end{lemma}

\begin{proof}
List the atoms below and above one in increasing order:
\[
 r_1<\cdots<r_p<1<t_1<\cdots<t_q.
\]
The deficits $\mu\{r_i\}(1-r_i)$ and surpluses
$\mu\{t_j\}(t_j-1)$ have equal totals.

Match deficit and surplus starting with $(r_1,t_q)$, exhausting one
side before moving inward. A matched amount $\theta$ between $r$ and
$t$ gives the component $\rho_{r,t}$ weight
\[
 \lambda=\frac{\theta}{1-r}+\frac{\theta}{t-1}.
\]
Its masses at $r,t$ are $\theta/(1-r),\theta/(t-1)$, respectively.
Together with $\mu\{1\}\delta_1$, the components reconstruct $\mu$.

The intervals $[r,t]$ produced by this procedure are nested. If
$r\le r'\le1\le t'\le t$, every mean-one law supported on $[r',t']$
is dominated in convex order by $\rho_{r,t}$. Indeed, the chord bound
\[
 \varphi(x)\le\frac{t-x}{t-r}\varphi(r)
                  +\frac{x-r}{t-r}\varphi(t),\qquad r\le x\le t,
\]
integrates to the asserted inequality. Order the components from the
innermost interval to the outermost and put $\delta_1$ first.
\end{proof}

\begin{proposition}[Extension to arbitrary laws]
\label{prop:local-all-law-squares}
If \eqref{eq:two-point-test} holds with constant $A$, then every
mean-one law $\mu$ on $[0,\infty)$ satisfies
\begin{equation}
 \Df(\mu\star\mu)\le2A\Df(\mu).
 \label{eq:local-all-law-square}
\end{equation}
Consequently, every pair of mean-one laws $\alpha,\beta$ satisfies
\begin{equation}
 \Df\!\left[\left(\frac{\alpha+\beta}{2}\right)^{\star2}\right]
 \le2A\Df(\alpha\star\beta).
 \label{eq:local-full-binary}
\end{equation}
\end{proposition}

\begin{proof}
The two-point bound extends to a law with an atom at zero by taking its
lower atom to zero. If $f(0)=\infty$, the right-hand side is already
infinite; otherwise convexity and lower semicontinuity give convergence
of all terms.

First let $\mu$ be finitely supported and use
Lemma~\ref{lem:local-nested-decomposition}. Multiplication by an
independent nonnegative variable preserves convex order, since
$x\mapsto\mathbb E\varphi(xZ)$ is convex. Thus, when $i\le j$,
\[
 \Df(\rho_i\star\rho_j)
 \le\Df(\rho_j\star\rho_j)\le A\Df(\rho_j).
\]
Expanding the mixture square gives
\begin{align*}
 \Df(\mu\star\mu)
 &\le A\sum_j\left(\lambda_j^2
                  +2\lambda_j\sum_{i<j}\lambda_i\right)\Df(\rho_j)\\
 &\le2A\sum_j\lambda_j\Df(\rho_j)=2A\Df(\mu).
\end{align*}

For an arbitrary law, take $X\sim\mu$ and increasing finite
sigma-fields $\mathcal G_k$ generating $\sigma(X)$. The conditional
means $X_k=\mathbb E[X\mid\mathcal G_k]$ have finite support, mean one,
and converge almost surely to $X$. Conditional Jensen gives
$\mathbb E f(X_k)\le\mathbb E f(X)$. Couple independent copies in the
same way. The finite-support bound and Fatou's lemma then yield
\[
 \mathbb E f(XX')
 \le\liminf_k\mathbb E f(X_kX_k')
 \le2A\mathbb E f(X).
\]

To prove \eqref{eq:local-full-binary}, put
$\nu=(\alpha+\beta)/2$. If independent $X\sim\alpha$ and $Y\sim\beta$
have mean one, then $X=\mathbb E[XY\mid X]$ and
$Y=\mathbb E[XY\mid Y]$. Hence
\[
 \Df(\nu)=\tfrac12[\Df(\alpha)+\Df(\beta)]
 \le\Df(\alpha\star\beta).
\]
Apply \eqref{eq:local-all-law-square} to $\nu$.
\end{proof}

\begin{proof}[Proof of Theorem~\ref{thm:main} and
Corollary~\ref{cor:scalar-main}]
The implication \textnormal{(i)}$\Rightarrow$\textnormal{(ii)} is immediate.
Lemma~\ref{lem:local-two-point-comparison} gives
\textnormal{(ii)}$\Rightarrow$\textnormal{(iii)} with $A=8K+2$,
as well as the equivalence with \eqref{eq:endpoint-test} and the
comparison $M\le A\le M+2$ for the least constants.

Assume \textnormal{(iii)}. Proposition~\ref{prop:local-all-law-squares}
gives \eqref{eq:B} with constant $2A$. Theorem~\ref{thm:analytic}
then gives \eqref{eq:H} with constant $\Cuniv(2A)$.
Propositions~\ref{thm:likelihood-laws} and~\ref{thm:lifting} give
\eqref{eq:homogenization} when $P_i\ll Q_i$;
Lemma~\ref{lem:regularization} removes this assumption. This proves
\textnormal{(iii)}$\Rightarrow$\textnormal{(i)} and the bound
$C_f\le\Cuniv(2A)$. Substituting $A=8K+2$ or $A=M+2$ gives the
other stated bounds. The implication from $A$ to the same Bernoulli
constant is part of Lemma~\ref{lem:local-two-point-comparison}.

At two coordinates on an arbitrary alphabet, the same argument uses
\eqref{eq:B} directly, giving constant $2A$. In particular, a
Bernoulli bound $K$ gives a two-coordinate bound $16K+4$ on every
alphabet. These estimates also show that finiteness of the scalar
criterion is independent of the supporting slope used to normalize
$f$.
\end{proof}

\section{Examples}
\label{sec:examples}

\subsection{Symmetric divergences}

\begin{corollary}
\label{cor:local-symmetric}
Suppose $f(t)=t f(1/t)$. The scalar condition
\eqref{eq:endpoint-test} holds if and only if there is
$C<\infty$ such that
\begin{equation}
 f(t^2)\le C(t+1)f(t),\qquad t>1.
 \label{eq:local-symmetric-doubling}
\end{equation}
\end{corollary}

\begin{proof}
Necessity follows by setting $r=1/t$ in
\eqref{eq:endpoint-test}: symmetry reduces that inequality
to \eqref{eq:local-symmetric-doubling} with $C=M$.

For sufficiency, symmetry extends the bound to
$f(r^2)\le C(1+r)f(r)$ for $0<r<1$. For the probabilities in
\eqref{eq:local-two-point-law}, both $a(1+r)$ and $b(1+t)$ are at most
$2$, since their sum equals $2$. Consequently,
\[
 a^2f(r^2)+b^2f(t^2)
 \le2C[af(r)+bf(t)].
\]
This is \eqref{eq:endpoint-test} with $M=2C$.
\end{proof}

\subsection{Power divergences}

\begin{proposition}
\label{prop:local-power-divergences}
For the Cressie--Read generators \cite{CressieRead1984}
\[
 f_\alpha(t)=
 \begin{cases}
 \displaystyle\frac{t^\alpha-\alpha t+\alpha-1}{\alpha(\alpha-1)},
       &\alpha\notin\{0,1\},\\[2mm]
 -\log t+t-1,&\alpha=0,\\
 t\log t-t+1,&\alpha=1,
 \end{cases}
\]
homogenization holds exactly when $0\le\alpha\le1$. Throughout this
interval it holds with constant $1$. Outside this interval it fails
already for two Bernoulli coordinates.
\end{proposition}

\begin{proof}
For $0<\alpha<1$, put
\[
 A_\alpha(P,Q)=\int p^\alpha q^{1-\alpha}\,\mathrm d\lambda.
\]
Then
\[
 D_{f_\alpha}(P\|Q)
 =\frac{1-A_\alpha(P,Q)}{\alpha(1-\alpha)}.
\]
Joint concavity, multiplicativity, and the arithmetic--geometric mean
inequality give, with $a_i=A_\alpha(P_i,Q_i)$,
\[
 A_\alpha(\bar P,\bar Q)^n
 \ge\left(\frac1n\sum_i a_i\right)^n
 \ge\prod_i a_i
 =A_\alpha\!\left(\bigotimes_iP_i,\bigotimes_iQ_i\right).
\]
This proves the claim with constant $1$. At $\alpha=1$, joint
convexity and additivity of Kullback--Leibler divergence give
\[
 nD(\bar P\|\bar Q)
 \le\sum_iD(P_i\|Q_i)
 =D\!\left(\bigotimes_iP_i\middle\|\bigotimes_iQ_i\right).
\]
Interchanging $P$ and $Q$ proves the endpoint $\alpha=0$.

If $\alpha>1$, fix $0<r<1$ and let $t\to\infty$ in the endpoint
criterion. The right-endpoint contribution divided by the original
two-point divergence grows on the order of $t^{\alpha-1}$.
If $\alpha<0$, fix $t>1$ and let $r\downarrow0$; the analogous
lower-endpoint ratio grows on the order of $r^\alpha$.
Both are unbounded, so Lemma~\ref{lem:local-two-point-comparison} rules
out even the Bernoulli inequality.
\end{proof}

\subsection{Total variation}

For $f(t)=|t-1|/2$, a product of two independent mean-one variables
with the same law satisfies
\[
 \E f(XX')
 \le\frac12\E\bigl[X|X'-1|+|X-1|\bigr]
 =2\E f(X).
\]
Theorem~\ref{thm:main} gives a universal homogenization constant.
More precisely, for $\nu=(\alpha+\beta)/2$, expansion of $\nu^{\star2}$
and the same product bound give
\begin{align*}
 \Df(\nu^{\star2})
 &\le\tfrac12\Df(\alpha)
       +\tfrac12\Df(\alpha\star\beta)+\tfrac12\Df(\beta)\\
 &\le\tfrac32\Df(\alpha\star\beta).
\end{align*}
Thus $\Cuniv(3/2)$ is admissible. A direct proof with an explicit
universal constant is given in \cite{Kontorovich2026TV}.

\subsection{Failure for $\chi^2$}

For $f(t)=(t-1)^2$, take
\[
 P_1=\delta_1,\qquad
 Q_1=(1-\varepsilon)\delta_0+\varepsilon\delta_1,
 \qquad P_2=Q_2=\delta_0,
 \qquad0<\varepsilon<1.
\]
The original product divergence is $(1-\varepsilon)/\varepsilon$.
In the averaged pair, symbol $1$ has likelihood ratio
$1/\varepsilon$ and $\bar Q$-probability $\varepsilon/2$.
Its repetition therefore contributes
\[
 \frac{\varepsilon^2}{4}
       (\varepsilon^{-2}-1)^2
 =\frac{(1-\varepsilon^2)^2}{4\varepsilon^2}
\]
to the homogenized divergence. The ratio of the two divergences is at
least
\[
 \frac{(1-\varepsilon)(1+\varepsilon)^2}{4\varepsilon}
 \longrightarrow\infty.
\]
\subsection{Insufficiency of endpoint growth bounds}

\begin{proposition}[A generator with a flat interval]
\label{ex:local-flat-generator}
For $p>2$, let
\[
 f(t)=
 \begin{cases}
 (1-t)^p,&0\le t<1,\\
 0,&1\le t\le2,\\
 t-2,&t>2.
 \end{cases}
\]
This generator is convex, bounded at zero, and linear at infinity, yet
it fails the two-coordinate Bernoulli inequality.
\end{proposition}

\begin{proof}
Fix $\sqrt2<t<2$, set $r=1-\varepsilon$, and use the two-point law
$\rho_{r,t}$. Its upper atom has probability of order $\varepsilon$,
while its lower atom has probability tending to one. Since $f(t)=0$,
\[
 \Df(\rho_{r,t})\asymp\varepsilon^p.
\]
In the square, the upper atom repeats with probability of order
$\varepsilon^2$, and $f(t^2)=t^2-2>0$. Hence
\[
 \frac{\Df(\rho_{r,t}\star\rho_{r,t})}{\Df(\rho_{r,t})}
 \gtrsim\varepsilon^{2-p}\longrightarrow\infty.
\]
\end{proof}

Thus asymptotic bounds at zero and infinity do not suffice: the
criterion also depends on the joint behavior on both sides of one.

\section{Quantitative dependence}
\label{sec:constants}

The modulus $\Cuniv(K)$ is uniform over generators with the same
two-coordinate bound. The two-point reduction loses a factor of two;
the main loss comes from controlling repeated sampling. The dependence
on $K$ is not asserted to be optimal. For nearby experiments, the bound
in Section~\ref{app:poisson} is $c^{O(\log\log c)}$, where
$c=\max\{2,4K-2\}$. The bound away from zero Hellinger distance is
polynomial in $c$, apart from the constant in Klass's theorem.
The final choice of $\Cuniv$ is given in
\eqref{eq:universal-modulus}.

\section{Product estimates and Hellinger coordinates}
\label{app:global}

We now prove Theorem~\ref{thm:analytic}. Throughout
Sections~\ref{app:global}--\ref{app:completion}, assume $f\ge0$,
$f(1)=0$, and \eqref{eq:B} with constant $K$.

Set
\begin{equation}
 c:=\max\{2,4K-2\},\qquad a:=\log_2c.
 \label{eq:c-a}
\end{equation}
Unless specified otherwise, constants depend only on $c$.
We first work with positive finitely supported likelihood laws;
Section~\ref{app:completion} removes this restriction with unchanged
constants.

\subsection{Product estimates}

\begin{lemma}[Convolution monotonicity]
\label{lem:monotonicity}
If $\alpha,\beta$ are mean-one laws, then
\[
 \Df(\alpha)\le\Df(\alpha\star\beta),
 \qquad
 \Df(\beta)\le\Df(\alpha\star\beta).
\]
The assertion remains valid in $[0,\infty]$.
\end{lemma}

\begin{proof}
Let $X\sim\alpha$ and $Y\sim\beta$ be independent. Since $\E Y=1$,
\[
 X=\E[XY\mid X].
\]
Conditional Jensen gives $f(X)\le\E[f(XY)\mid X]$. Taking expectations
proves the first inequality, and the second is symmetric.
\end{proof}

\begin{proposition}[Products and repeated factors]
\label{prop:binary-calculus}
For mean-one laws $\alpha,\beta,\rho,\tau$ and every integer $m\ge1$,
\begin{align}
 \Df(\alpha^{\star2})+\Df(\beta^{\star2})
 &\le c\,\Df(\alpha\star\beta),\label{eq:two-square}\\
 \Df(\rho^{\star2})
 &\le c\,\Df(\rho),\label{eq:self-square}\\
 \Df(\rho\star\tau)
 &\le c\,[\Df(\rho)+\Df(\tau)],\label{eq:reverse-product}\\
 \Df(\rho^{\star m})
 &\le c^{\lceil\log_2m\rceil}\Df(\rho)
 \le c\,m^a\Df(\rho).\label{eq:power-growth}
\end{align}
\end{proposition}

\begin{proof}
By linearity,
\[
 \Df\!\left[
  \left(\frac{\alpha+\beta}{2}\right)^{\star2}
 \right]
 =
 \frac14\Df(\alpha^{\star2})
 +\frac12\Df(\alpha\star\beta)
 +\frac14\Df(\beta^{\star2}).
\]
Thus \eqref{eq:B} yields \eqref{eq:two-square}. Taking
$\beta=\delta_1$ proves \eqref{eq:self-square}.

For $v=(\rho+\tau)/2$,
\[
 \frac12\Df(\rho\star\tau)
 \le \Df(v^{\star2})
 \le c\Df(v)
 =\frac c2[\Df(\rho)+\Df(\tau)],
\]
which proves \eqref{eq:reverse-product}.

Iterating \eqref{eq:self-square} gives
$\Df(\rho^{\star2^r})\le c^r\Df(\rho)$. If $m\le2^r$, then
$\rho^{\star m}$ is a convolution factor of $\rho^{\star2^r}$, so
Lemma~\ref{lem:monotonicity} applies. Take
$r=\lceil\log_2m\rceil$.
\end{proof}

\subsection{Hellinger transform}

For a positive mean-one law $\mu$, define a finite measure $\eta_\mu$ on
$\R$ by
\begin{equation}
 \int_\R\varphi(x)\,\eta_\mu(\dd x)
 :=
 \int_{(0,\infty)}
 \varphi\!\left(\frac12\log t\right)\sqrt t\,\mu(\dd t).\label{eq:hell-def}
\end{equation}
Let
\[
 \mathcal A
 :=
 \left\{
  \eta\ge0:
  \int e^x\,\eta(\dd x)=\int e^{-x}\,\eta(\dd x)=1
 \right\}.
\]

\begin{proposition}[Hellinger correspondence]
\label{prop:hellinger}
The map $\mu\mapsto\eta_\mu$ is an affine bijection from positive
mean-one laws onto $\mathcal A$. Its inverse is
\[
 \int\psi(t)\,\mu_\eta(\dd t)
 =
 \int\psi(e^{2x})e^{-x}\,\eta(\dd x).
\]
Moreover,
\[
 \eta_{\mu\star\nu}=\eta_\mu*\eta_\nu,
\]
where $*$ denotes additive convolution. If
\begin{equation}
 h(x):=h_f(x):=e^{-x}f(e^{2x}),
 \qquad
 \Ff(\eta):=\int h\,\dd\eta,\label{eq:h-def}
\end{equation}
then
\[
 \Df(\mu)=\Ff(\eta_\mu).
\]
Finally,
\begin{equation}
 d(\mu)
 :=1-\eta_\mu(\R)
 =1-\int\sqrt t\,\mu(\dd t)
 =\frac12\int(\sqrt t-1)^2\,\mu(\dd t).\label{eq:defect}
\end{equation}
\end{proposition}

\begin{proof}
The two exponential moments of $\eta_\mu$ are respectively
$\int t\,\dd\mu=1$ and $\int1\,\dd\mu=1$. Conversely, the displayed
inverse has total mass $\int e^{-x}\dd\eta=1$ and mean
$\int e^x\dd\eta=1$. The remaining identities follow from the
definitions.
\end{proof}

\begin{lemma}[Convexity of the transformed generator]
\label{lem:h-convex}
The function $h$ in \eqref{eq:h-def} is finite, nonnegative, convex, and
satisfies $h(0)=0$.
\end{lemma}

\begin{proof}
Only convexity requires proof. If $f$ is smooth and $t=e^{2x}$, then
\[
 h''(x)=t^{-1/2}\bigl[f(t)+4t^2f''(t)\bigr]\ge0.
\]
For general finite convex $f$, fix a compact interval $J\subset\R$ and
let $I:=\{e^{2x}:x\in J\}\Subset(0,\infty)$. Extend $f$ from a slightly
larger compact interval containing $I$ to a finite convex function on
$\R$ by continuing it with supporting affine rays at the endpoints, and
convolve that extension with standard smooth mollifiers. This produces
smooth convex functions $f_r$ converging uniformly to $f$ on $I$. If
$\varepsilon_r:=\sup_I|f_r-f|$, replace $f_r$ by
$f_r+\varepsilon_r$; it is then nonnegative on $I$ and still converges
uniformly to $f$ there. The corresponding functions
$h_r(x)=e^{-x}f_r(e^{2x})$ are convex on $J$ by the displayed second
derivative formula and converge uniformly to $h$ on $J$. Hence $h$ is
convex on every compact interval, and therefore on $\R$.
\end{proof}

For $\eta\in\mathcal A$, the measure
\[
 M_\eta(\dd x):=\cosh x\,\eta(\dd x)
\]
is a probability law. If $X\sim M_\eta$ and $U=\tanh X$, then
$\E U=\int\sinh x\,\eta(\dd x)=0$. Define
\begin{equation}
 \Phi_f(u)
 :=
 (1-u)f\!\left(\frac{1+u}{1-u}\right),
 \qquad -1<u<1.\label{eq:posterior-Phi}
\end{equation}
It is nonnegative and satisfies $\Phi_f(0)=0$. Convexity follows by
restricting the perspective
\[
 (a,b)\longmapsto b f(a/b),
 \qquad b>0,
\]
to $(a,b)=(1+u,1-u)$. For $b=\theta b_1+(1-\theta)b_2$ and
$\lambda=\theta b_1/b$,
\[
 \frac{\theta a_1+(1-\theta)a_2}{b}
 =
 \lambda\frac{a_1}{b_1}
 +(1-\lambda)\frac{a_2}{b_2},
\]
so Jensen's inequality, multiplied by $b$, proves convexity.

\begin{lemma}[Posterior identities]
\label{lem:posterior-identities}
For $\eta\in\mathcal A$ and the associated centered score $U$,
\begin{equation}
 \Ff(\eta)=\E\Phi_f(U).\label{eq:posterior-energy}
\end{equation}
If $d=1-\eta(\R)$, then
\begin{equation}
 \frac12\E U^2\le d\le\E U^2.\label{eq:posterior-defect}
\end{equation}
\end{lemma}

\begin{proof}
Since
\[
 1-\tanh x=\frac{e^{-x}}{\cosh x},
 \qquad
 \frac{1+\tanh x}{1-\tanh x}=e^{2x},
\]
one has $\cosh x\,\Phi_f(\tanh x)=h(x)$, proving
\eqref{eq:posterior-energy}. Also
\[
 d
 =\int(\cosh x-1)\,\eta(\dd x)
 =\E\bigl[1-\sqrt{1-U^2}\bigr].
\]
Now use
\[
 \frac{u^2}{2}
 \le1-\sqrt{1-u^2}
 =\frac{u^2}{1+\sqrt{1-u^2}}
 \le u^2.
\]
\end{proof}

Define the symmetric function
\begin{equation}
 \mathcal M_f(s)
 :=
 \Phi_f(\tanh s)+\Phi_f(-\tanh s),
 \qquad s\ge0.\label{eq:M-profile}
\end{equation}

\begin{lemma}[Symmetric doubling bound]
\label{lem:symmetric-moderation}
The function $\mathcal M_f$ is nondecreasing and
\begin{equation}
 \mathcal M_f(2s)\le2c\,\mathcal M_f(s),
 \qquad s\ge0.\label{eq:M-doubling}
\end{equation}
Consequently, with $\gamma_c:=\log_2(2c)$,
\begin{equation}
 \mathcal M_f(\lambda s)
 \le2c\,\lambda^{\gamma_c}\mathcal M_f(s),
 \qquad \lambda\ge1.\label{eq:M-moderate}
\end{equation}
\end{lemma}

\begin{proof}
The function $u\mapsto\Phi_f(u)+\Phi_f(-u)$ is even and convex, hence
nondecreasing on $[0,1)$. For $s\ge0$, put
\[
 \omega_s:=\frac{\delta_s+\delta_{-s}}{2\cosh s}\in\mathcal A.
\]
Then
\[
 \Ff(\omega_s)=\frac12\mathcal M_f(s),
 \qquad
 \Ff(\omega_s^{*2})
 =\frac{\cosh(2s)}{4\cosh^2s}\mathcal M_f(2s).
\]
Apply \eqref{eq:self-square} in Hellinger coordinates and use
$\cosh(2s)\ge\cosh^2s$. This proves \eqref{eq:M-doubling}; dyadic
iteration gives \eqref{eq:M-moderate}.
\end{proof}

\begin{lemma}[Reference scale]
\label{lem:canonical-scale}
Put
\begin{equation}
 s_*:=\frac12\log2,
 \qquad
 m_f:=\mathcal M_f(s_*).
 \label{eq:canonical-scale}
\end{equation}
Then
\begin{equation}
 m_f=\frac23f(2)+\frac43f(1/2)>0,
 \label{eq:canonical-scale-formula}
\end{equation}
and hence
\begin{equation}
 f(1/2)\le\frac34m_f,
 \qquad
 f(2)\le\frac32m_f.
 \label{eq:canonical-scale-bounds}
\end{equation}
\end{lemma}

\begin{proof}
Since $\tanh s_*=1/3$, the identity in
\eqref{eq:canonical-scale-formula} follows directly from the definition
of $\Phi_f$. Suppose that $m_f=0$. Nonnegativity then gives
$f(1/2)=f(2)=0$. Let
\[
 \rho:=\frac23\delta_{1/2}+\frac13\delta_2.
\]
This is a mean-one law and $\Df(\rho)=0$. Iterating
\eqref{eq:self-square} gives $\Df(\rho^{\star2^r})=0$ for every $r$.
The extreme atoms of this law are $2^{-2^r}$ and $2^{2^r}$, both with
positive mass, so
\[
 f(2^{-2^r})=f(2^{2^r})=0.
\]
The zero set of a nonnegative convex function is an interval; these
intervals exhaust $(0,\infty)$. Lower semicontinuity at zero then gives
$f\equiv0$, contrary to the standing assumption. Thus $m_f>0$, and
\eqref{eq:canonical-scale-bounds} follows from
\eqref{eq:canonical-scale-formula}.
\end{proof}

\section{Comparison away from zero Hellinger distance}\label{app:separated}

Klass's inequalities yield a comparison with an additive error.
The normalization $\E e^{-X_i}=1$ controls cancellation, and a positive
lower bound on the total squared Hellinger distance absorbs the error.

\subsection{Growth of logarithmic transforms}

For $x\ge0$, define
\begin{equation}
 g_+(x):=e^{-x}f(e^x),
 \qquad
 g_-(x):=f(e^{-x}),
 \qquad
 \vartheta(x):=\frac{x}{1+x}.\label{eq:log-profiles}
\end{equation}
Both $g_+$ and $g_-$ are finite, continuous, nonnegative, and
nondecreasing. Indeed, $f$ is nonincreasing on $(0,1]$, while for
$t\ge1$,
\[
 \frac{f(t)}t
 =
 \frac{t-1}{t}\,
 \frac{f(t)-f(1)}{t-1},
\]
and both factors are nonnegative and nondecreasing.

\begin{lemma}[Doubling after bounded correction]
\label{lem:log-moderation}
For every $x\ge0$,
\begin{align}
 g_+(2x)
 &\le2c\,g_+(x)+8c f(1/2)\vartheta(x),
 \label{eq:gplus-doubling}\\
 g_-(2x)
 &\le2c\,g_-(x)+4c f(2)\vartheta(x).
 \label{eq:gminus-doubling}
\end{align}
Define
\begin{equation}
 \Psi_+(x):=g_+(x)+m_f\vartheta(x),
 \qquad
 \Psi_-(x):=g_-(x)+m_f\vartheta(x),
 \label{eq:normalized-profiles}
\end{equation}
and put
\begin{equation}
 D_c:=6c+2,
 \qquad
 \alpha_c:=2\log_2D_c.
 \label{eq:Dc-alphac}
\end{equation}
Then $\Psi_\pm$ are continuous and strictly increasing,
\begin{equation}
 \Psi_\pm(2x)\le D_c\Psi_\pm(x),
 \qquad x\ge0,
 \label{eq:Psi-doubling}
\end{equation}
and consequently
\begin{equation}
 \Psi_\pm(\lambda x)
 \le\lambda^{\alpha_c}\Psi_\pm(x),
 \qquad x\ge0,\quad\lambda\ge2.
 \label{eq:Psi-pm-moderate}
\end{equation}
Moreover,
\begin{equation}
 g_\pm\le\Psi_\pm\le g_\pm+m_f.
 \label{eq:Psi-sandwich}
\end{equation}
\end{lemma}

\begin{proof}
Put $t=e^x\ge1$, $\lambda=(2t-1)^{-1}$, and
\[
 \rho=\lambda\delta_t+(1-\lambda)\delta_{1/2}.
\]
This law has mean one. The $t^2$ atom of $\rho^{\star2}$ and
\eqref{eq:self-square} give
\[
 \lambda^2f(t^2)
 \le c\bigl[\lambda f(t)+(1-\lambda)f(1/2)\bigr].
\]
After division by $\lambda^2t^2$, use
\[
 \frac1{\lambda t}=2-\frac1t\le2,
 \qquad
 \frac{1-\lambda}{\lambda^2t^2}
 =\frac{2(t-1)(2t-1)}{t^2}
 \le4(1-e^{-x}),
\]
and
$1-e^{-x}\le\min\{x,1\}\le2x/(1+x)$, proving
\eqref{eq:gplus-doubling}.

For $t=e^{-x}\le1$, take $\lambda=(2-t)^{-1}$ and
$\rho=\lambda\delta_t+(1-\lambda)\delta_2$. Then
\[
 \frac1\lambda\le2,
 \qquad
 \frac{1-\lambda}{\lambda^2}
 =(1-t)(2-t)
 \le2(1-e^{-x})\le4\vartheta(x),
\]
which proves \eqref{eq:gminus-doubling}. By
Lemma~\ref{lem:canonical-scale},
\[
 8c f(1/2)\le6c m_f,
 \qquad
 4c f(2)\le6c m_f.
\]
Since $\vartheta(2x)\le2\vartheta(x)$,
\[
 \Psi_\pm(2x)
 \le2c\,g_\pm(x)+(6c+2)m_f\vartheta(x)
 \le D_c\Psi_\pm(x).
\]
The functions are strictly increasing because $m_f>0$ and
$\vartheta$ is strictly increasing. If $\lambda\ge2$ and
$r=\lceil\log_2\lambda\rceil$, then
\[
 \Psi_\pm(\lambda x)
 \le D_c^r\Psi_\pm(x)
 \le D_c\lambda^{\log_2D_c}\Psi_\pm(x)
 \le\lambda^{2\log_2D_c}\Psi_\pm(x),
\]
This proves \eqref{eq:Psi-pm-moderate};
\eqref{eq:Psi-sandwich} follows from $0\le\vartheta\le1$.
\end{proof}

\subsection{Klass's moment inequalities}

We use Klass~\cite[definitions on pp.~414--415, Theorem~3.1 on
p.~417, and Lemma~6.1 on p.~422]{Klass1981}, with the family of
partial sums reduced to the full sum.
Let $\Psi:[0,\infty)\to[0,\infty)$ be continuous and strictly
increasing, with $\Psi(0)=0$, and suppose that
\begin{equation}
 \Psi(\lambda x)\le\lambda^\alpha\Psi(x),
 \qquad x\ge0,\quad\lambda\ge2.\label{eq:Psi-moderate}
\end{equation}
Extend $\Psi$ evenly to $\R$. For independent real variables
$Z_1,\ldots,Z_N$, $y>0$, and $a\in\R$, define
$K_{\mathbf Z}(y)$ by the equation
\begin{equation}
 y\sum_{i=1}^N
 \E\!\bigl[Z_i^2\one_{\{|Z_i|\le k\}}\bigr]
 +
 y\frac{k^2}{\Psi(k)}
 \sum_{i=1}^N
 \E\!\bigl[\Psi(|Z_i|)\one_{\{|Z_i|>k\}}\bigr]
 =k^2,\label{eq:K-functional}
\end{equation}
with the standard endpoint conventions, and set
\begin{equation}
 M_{\mathbf Z}(y,a)
 :=
 \sup\left\{
  u\ge0:
  y\left|
   a+\sum_{i=1}^N\E[Z_i\one_{\{|Z_i|\le u\}}]
  \right|\ge u
 \right\}.\label{eq:M-functional}
\end{equation}
Klass proves that there is $B_\alpha<\infty$, depending only on
$\alpha$, such that
\begin{equation}
 \E\Psi\!\left(\left|a+\sum_{i=1}^NZ_i\right|\right)
 \le
 B_\alpha\,
 \Psi\!\left(K_{\mathbf Z}(1)\vee M_{\mathbf Z}(1,a)\right),\label{eq:Klass-upper}
\end{equation}
and
\begin{equation}
 \Psi(M_{\mathbf Z}(1,a))
 \le
 \max\left\{
  \Psi(K_{\mathbf Z}(2)),
  2^{\alpha+2}
  \E\Psi\!\left(\left|a+\sum_{i=1}^NZ_i\right|\right)
 \right\}.\label{eq:Klass-M}
\end{equation}
Also $K_{\mathbf Z}(y)$ is nondecreasing in $y$.

\subsection{Homogenization of logarithmic sums}

\begin{theorem}[Comparison with the average law]
\label{thm:exp-regular-cover}
For every $\alpha>0$ there is a finite constant $R_\alpha$, depending
only on $\alpha$, with the following property. Let
$\Psi:[0,\infty)\to[0,\infty)$ be continuous and strictly increasing,
with $\Psi(0)=0$, and suppose that
\eqref{eq:Psi-moderate} holds. Let
$X_1,\ldots,X_N$ be independent finitely supported real variables such
that
\begin{equation}
 \E e^{-X_i}=1,
 \qquad 1\le i\le N.
 \label{eq:exp-normalization}
\end{equation}
Let $\widetilde X_1,\ldots,\widetilde X_N$ be iid with common law
\[
 \frac1N\sum_{i=1}^N\Law(X_i).
\]
Then
\begin{equation}
 \E\Psi\!\left(\pos{\sum_{j=1}^N\widetilde X_j}\right)
 \le
 R_\alpha\left[
  \Psi(8)+
  \E\Psi\!\left(\pos{\sum_{i=1}^NX_i}\right)
 \right].
 \label{eq:exp-RC}
\end{equation}
The same $R_\alpha$ works for every admissible $\Psi$.
\end{theorem}

\begin{proof}
Suppose first that $\Psi$ is unbounded. Write
\[
 S:=\sum_{i=1}^NX_i,
 \qquad
 \widetilde S:=\sum_{j=1}^N\widetilde X_j.
\]
For every nonnegative Borel function $H$, and for every real Borel
function $H$ integrable under all the laws involved,
\begin{equation}
 \sum_{j=1}^N\E H(\widetilde X_j)
 =
 \sum_{i=1}^N\E H(X_i).
 \label{eq:aggregate-marginals}
\end{equation}
Hence
\begin{equation}
 K_{\widetilde{\mathbf X}}(y)=K_{\mathbf X}(y),
 \qquad
 M_{\widetilde{\mathbf X}}(y,0)=M_{\mathbf X}(y,0).
 \label{eq:KM-equality}
\end{equation}
Equations \eqref{eq:Klass-upper}, \eqref{eq:Klass-M}, and monotonicity
of $K$ therefore give
\begin{equation}
 \E\Psi(|\widetilde S|)
 \le
 B_\alpha
 \max\left\{
  \Psi(K_{\mathbf X}(2)),
  2^{\alpha+2}\E\Psi(|S|)
 \right\}.
 \label{eq:RC-radial}
\end{equation}

Put $b:=K_{\mathbf X}(2)$. If $b<8$, then
$\Psi(b)\le\Psi(8)$. Assume henceforth that $b\ge8$, and define
\begin{equation}
 A_b:=\sum_{i=1}^N\E[X_i^2\one_{\{|X_i|\le b\}}],
 \qquad
 T_b:=\sum_{i=1}^N\E[\Psi(|X_i|)\one_{\{|X_i|>b\}}].
 \label{eq:Ab-Tb}
\end{equation}
By definition of $b$,
\begin{equation}
 A_b+\frac{b^2}{\Psi(b)}T_b=\frac{b^2}{2}.
 \label{eq:K2-identity}
\end{equation}
Let
\[
 q(t):=\sum_{i=1}^N\Pp(|X_i|>t).
\]
Then
\begin{equation}
 q(b)\le\frac{T_b}{\Psi(b)}\le\frac12.
 \label{eq:q-half}
\end{equation}

For $t\ge b$, we claim that
\begin{equation}
 q(t)\le4\Pp(|S|>t/4).
 \label{eq:tail-comparison}
\end{equation}
Write $R_i=S-X_i$. Independence and
\eqref{eq:exp-normalization} give $\E e^{-R_i}=1$. Let
$q^\pm(t)=\sum_i\Pp(\pm X_i>t)$. For the positive tail, put
\[
 E_i^+:=\{X_i>t,\ R_i\ge-3t/4\}.
\]
Markov's inequality gives
$\Pp(R_i<-3t/4)\le e^{-3t/4}\le1/4$, and hence
\[
 \sum_i\Pp(E_i^+)\ge\frac34q^+(t).
\]
Moreover,
$E_i^+\cap E_j^+\subseteq\{X_i>t,X_j>t\}$, so independence and
\eqref{eq:q-half} imply
\[
 \sum_{i<j}\Pp(E_i^+\cap E_j^+)
 \le\frac12q^+(t)^2\le\frac14q^+(t).
\]
Bonferroni's inequality therefore gives
\begin{equation}
 \Pp(S>t/4)\ge\frac12q^+(t).
 \label{eq:positive-tail}
\end{equation}

For the negative tail, suppose first that
$\Pp(R_i>3t/4)>1/2$ for some $i$. Since
$\E e^{-X_i}=1$,
$\Pp(X_i<-\log2)\le1/2$. Independence and $t\ge8$ yield
\[
 \Pp(S>t/2)
 \ge
 \Pp(R_i>3t/4,\ X_i\ge-\log2)>\frac14.
\]
Together with $q(t)\le1/2$, this proves \eqref{eq:tail-comparison} in
this case. Otherwise
$\Pp(R_i>3t/4)\le1/2$ for every $i$. Put
\[
 E_i^-:=\{X_i<-t,\ R_i\le3t/4\}.
\]
Then
\[
 \sum_i\Pp(E_i^-)\ge\frac12q^-(t),
 \qquad
 \sum_{i<j}\Pp(E_i^-\cap E_j^-)
 \le\frac12q^-(t)^2\le\frac14q^-(t),
\]
so
\begin{equation}
 \Pp(S<-t/4)\ge\frac14q^-(t).
 \label{eq:negative-tail}
\end{equation}
Equations \eqref{eq:positive-tail} and \eqref{eq:negative-tail} prove
\eqref{eq:tail-comparison}.

Integration against $\dd\Psi$ and \eqref{eq:tail-comparison} give
\begin{align}
 T_b
 &=
 \Psi(b)q(b)+\int_{(b,\infty)}q(t)\,\dd\Psi(t)\notag\\
 &\le4\E\Psi(4|S|)
 \le4^{\alpha+1}\E\Psi(|S|).
 \label{eq:Tb-tail}
\end{align}

Fix a numerical $q_0\in(0,1/4)$, to be chosen below. If
$T_b\ge q_0\Psi(b)$, then
\begin{equation}
 \Psi(b)
 \le\frac{4^{\alpha+1}}{q_0}\E\Psi(|S|).
 \label{eq:scale-tail-case}
\end{equation}
Suppose instead that $T_b<q_0\Psi(b)$. Then
\begin{equation}
 q(b)<q_0,
 \qquad
 A_b>\left(\frac12-q_0\right)b^2\ge\frac14b^2.
 \label{eq:variance-reserve}
\end{equation}
Set
\[
 Z_i:=X_i\one_{\{|X_i|\le b\}},
 \quad
 a_i:=\E Z_i,
 \quad
 a:=\sum_i a_i,
 \quad
 W:=\sum_i(Z_i-a_i),
 \quad
 v^2:=\E W^2.
\]
Since $e^{-x}\ge1-x$,
\[
 1=\E e^{-X_i}
 \ge\E[(1-X_i)\one_{\{|X_i|\le b\}}],
\]
and hence
\begin{equation}
 \sum_i a_i^-\le q(b)<q_0.
 \label{eq:negative-mean-budget}
\end{equation}
Also
\begin{equation}
 A_b=v^2+\sum_i a_i^2.
 \label{eq:variance-mean-split}
\end{equation}
If $v^2\ge A_b/2$, then $v^2\ge b^2/8$. Otherwise
$\sum_i a_i^2\ge b^2/8$. By
\eqref{eq:negative-mean-budget}, for $b\ge8$ and $q_0\le1/4$,
\[
 \sum_i(a_i^+)^2\ge\frac{b^2}{16},
 \qquad
 \sum_i a_i^+\ge\frac b4,
 \qquad
 a\ge\frac b8.
\]
Thus, with $Y:=a+W=\sum_iZ_i$,
\begin{equation}
 m_2:=\E Y^2=a^2+v^2\ge\frac{b^2}{64}.
 \label{eq:Y-second}
\end{equation}
The centered summands $Z_i-a_i$ are bounded by $2b$, so
\begin{equation}
 \E W^4\le4b^2v^2+3v^4.
 \label{eq:W-fourth}
\end{equation}
Since $a^2,v^2\le m_2$ and $b^2\le64m_2$, the inequality
$(a+W)^4\le8a^4+8W^4$ gives
\begin{equation}
 \E Y^4\le C_*m_2^2
 \label{eq:Y-fourth}
\end{equation}
for a numerical $C_*$. Moreover,
\[
 \E[Y^2\one_{\{Y^2\ge m_2/2\}}]\ge m_2/2.
\]
Cauchy--Schwarz therefore gives
\[
 \Pp(Y^2\ge m_2/2)
 \ge\frac{m_2^2}{4\E Y^4}
 \ge\frac1{4C_*}.
\]
Hence there are numerical constants $c_1,c_2>0$ such that
\begin{equation}
 \Pp(|Y|\ge c_1b)\ge c_2.
 \label{eq:Y-large}
\end{equation}
Choose $q_0<\min\{1/4,c_2/2\}$. On the event
$\bigcap_i\{|X_i|\le b\}$ one has $S=Y$, while the complementary event
has probability below $q_0$. Therefore
\[
 \Pp(|S|\ge c_1b)\ge c_2/2.
\]
Let $\lambda_1:=\max\{2,c_1^{-1}\}$. Monotonicity and
\eqref{eq:Psi-moderate} give
\[
 \Psi(b)\le\Psi(\lambda_1c_1b)
 \le\lambda_1^\alpha\Psi(c_1b)
 \le\frac{2\lambda_1^\alpha}{c_2}\E\Psi(|S|).
\]
Combining the two cases with the bound for $b<8$, we obtain
$C_{1,\alpha}<\infty$, depending only on $\alpha$, such that
\begin{equation}
 \Psi(K_{\mathbf X}(2))
 \le C_{1,\alpha}\bigl[\Psi(8)+\E\Psi(|S|)\bigr].
 \label{eq:K2-reserve}
\end{equation}

Finally, \eqref{eq:exp-normalization} gives $\E e^{-S}=1$, and hence
\[
 \Pp(S^->t)\le e^{-t}.
\]
Thus $S^-$ is stochastically dominated by an exponential random
variable $E$ of mean one. For $0\le x\le16$,
$\Psi(x)\le\Psi(16)\le2^\alpha\Psi(8)$, while for $x>16$,
\[
 \Psi(x)\le(x/8)^\alpha\Psi(8).
\]
Consequently there is $C_{2,\alpha}<\infty$, depending only on
$\alpha$, such that
\begin{equation}
 \E\Psi(S^-)
 \le\E\Psi(E)
 \le C_{2,\alpha}\Psi(8).
 \label{eq:negative-reserve}
\end{equation}
Since
$\Psi(|S|)=\Psi(S^+)+\Psi(S^-)$, insertion of
\eqref{eq:K2-reserve} and \eqref{eq:negative-reserve} into
\eqref{eq:RC-radial} proves \eqref{eq:exp-RC}, with a constant
$R_\alpha$ depending only on $\alpha$.

If $\Psi$ is bounded, apply the unbounded result to
\[
 \Psi_\varepsilon(x):=\Psi(x)+\varepsilon x^\alpha,
 \qquad \varepsilon>0.
\]
The same growth bound holds, and finite support permits passage to the
limit $\varepsilon\downarrow0$ with unchanged $R_\alpha$.
\end{proof}

\subsection{Divergence comparison}

\begin{theorem}[Comparison with additive error]
\label{thm:additive-RC}
Let $D_c$ and $\alpha_c$ be as in \eqref{eq:Dc-alphac}, and let
$R_{\alpha_c}$ be the universal constant from
Theorem~\ref{thm:exp-regular-cover}. Every positive finitely supported
mean-one family satisfies
\begin{equation}
 \Df\!\left[
  \left(\frac1n\sum_{i=1}^n\mu_i\right)^{\star n}
 \right]
 \le
 R_{\alpha_c}\left[
  \Df(\mu_1\star\cdots\star\mu_n)
  +(4D_c^4+2)m_f
 \right].
 \label{eq:additive-RC}
\end{equation}
\end{theorem}

\begin{proof}
Let $T_i$ be the coordinate variable on $(0,\infty)$ and define
probability measures
\[
 Q_i(\dd t):=\mu_i(\dd t),
 \qquad
 P_i(\dd t):=t\,\mu_i(\dd t).
\]
Under $P_i$, put $X_i^+=\log T_i$; under $Q_i$, put
$X_i^-=-\log T_i$. Then
\[
 \E_{P_i}e^{-X_i^+}=1,
 \qquad
 \E_{Q_i}e^{-X_i^-}=1.
\]
Writing $P=\otimes_iP_i$ and $Q=\otimes_iQ_i$, and splitting according
to the sign of $\sum_i\log T_i$, gives
\begin{equation}
 \begin{aligned}
 \Df(\mu_1\star\cdots\star\mu_n)
 ={}&
 \E_P g_+\!\left(\pos{\sum_iX_i^+}\right)\\
 &+
 \E_Q g_-\!\left(\pos{\sum_iX_i^-}\right).
 \end{aligned}
 \label{eq:linked-Esscher}
\end{equation}

Put $\overline\mu=n^{-1}\sum_i\mu_i$ and define
$\overline Q=\overline\mu$ and
$\overline P(\dd t)=t\overline\mu(\dd t)$. The logarithmic variables
under $\overline P^{\otimes n}$ and $\overline Q^{\otimes n}$ are the
iid samples from the average laws of the two arrays above. Apply
Theorem~\ref{thm:exp-regular-cover} to $\Psi_+$ and $\Psi_-$ with the
common exponent $\alpha_c$. By \eqref{eq:Psi-sandwich},
\[
 \E\Psi_\pm(S^+)\le\E g_\pm(S^+)+m_f.
\]
Moreover, at $x=\log2$,
\[
 g_+(\log2)=\frac12f(2)\le\frac34m_f,
 \qquad
 g_-(\log2)=f(1/2)\le\frac34m_f,
\]
and therefore
\[
 \Psi_\pm(\log2)\le\frac74m_f.
\]
Since $8<16\log2$, four applications of \eqref{eq:Psi-doubling} give
\begin{equation}
 \Psi_+(8)+\Psi_-(8)
 \le\frac72D_c^4m_f
 \le4D_c^4m_f.
 \label{eq:Psi-eight}
\end{equation}
The corresponding identity for $\overline\mu$, together with
$g_\pm\le\Psi_\pm$, bounds the homogenized divergence by
\[
 R_{\alpha_c}\left[
  \Df(\mu_1\star\cdots\star\mu_n)
  +(4D_c^4+2)m_f
 \right].
\]
This is \eqref{eq:additive-RC}.
\end{proof}

\begin{theorem}[Comparison at positive Hellinger distance]
\label{thm:large-defect}
For $d_0>0$, put
\begin{equation}
 r_{d_0}:=\frac{1-e^{-d_0}}4,
 \qquad
 s_{d_0}:=\operatorname{arctanh}(r_{d_0}),
 \qquad
 j_{d_0}:=\left\lceil
  \log_2\frac{s_*}{s_{d_0}}
 \right\rceil_+,
 \label{eq:large-defect-scales}
\end{equation}
where $\lceil x\rceil_+:=\max\{0,\lceil x\rceil\}$. If
\[
 \sum_{i=1}^n d(\mu_i)\ge d_0,
\]
then
\begin{equation}
 \Df(\overline\mu^{\star n})
 \le
 C_{\mathrm{LD}}(c,d_0)
 \Df(\mu_1\star\cdots\star\mu_n),
 \label{eq:large-defect}
\end{equation}
where
\begin{equation}
 C_{\mathrm{LD}}(c,d_0)
 :=
 R_{\alpha_c}\left[
  1+(4D_c^4+2)(2c)^{j_{d_0}}
 \right].
 \label{eq:CLD-general}
\end{equation}
In particular, for $d_0=1/4$ one has $j_{1/4}\le3$, and hence
\begin{equation}
 C_{\mathrm{LD}}(c,1/4)
 \le
 R_{\alpha_c}\left[
  1+8c^3(4D_c^4+2)
 \right].
 \label{eq:CLD-quarter}
\end{equation}
\end{theorem}

\begin{proof}
Use the product experiment $(P,Q)$ from the proof of
Theorem~\ref{thm:additive-RC}. Its Hellinger affinity is
\[
 \Aff(P,Q)
 =
 \prod_i\int\sqrt t\,\mu_i(\dd t)
 =
 \prod_i[1-d(\mu_i)]
 \le e^{-d_0}.
\]
Therefore
\[
 \TV(P,Q)\ge1-\Aff(P,Q)\ge1-e^{-d_0}.
\]
Take a measurable set $A$ such that
$P(A)-Q(A)=\TV(P,Q)$. Apply the channel which, with probability $1/2$,
outputs an independent fair bit and, with probability $1/2$, outputs
$\one_A$. Let $P',Q'$ be the resulting Bernoulli laws. Their parameters
belong to $[1/4,3/4]$, and
\[
 \TV(P',Q')=\frac12\TV(P,Q).
\]
Let $U$ be the centered posterior score of the binary experiment
$(P',Q')$. Under the equal-prior mixture,
\[
 \E U_+=\E U_-=\frac12\TV(P',Q')
 \ge r_{d_0}.
\]
Writing $r:=\E U_+$, let
\[
 p_+:=\Pp(U>0),
 \qquad
 u_+:=\E[U\mid U>0].
\]
Then $r=p_+u_+$. Since $\Phi_f$ is convex and $\Phi_f(0)=0$,
\[
 \Phi_f(r)
 \le p_+\Phi_f(u_+)
 \le\E[\Phi_f(U)\one_{\{U>0\}}].
\]
Applying the same argument to $-U$ and using
Lemma~\ref{lem:posterior-identities},
\[
 \mathcal M_f(\operatorname{arctanh}r)
 =\Phi_f(r)+\Phi_f(-r)
 \le\E\Phi_f(U)
 =D_f(P'\|Q').
\]
By monotonicity of $\mathcal M_f$, data processing, and
\eqref{eq:M-doubling},
\begin{align}
 m_f
 &=\mathcal M_f(s_*)\notag\\
 &\le(2c)^{j_{d_0}}\mathcal M_f(s_{d_0})\notag\\
 &\le(2c)^{j_{d_0}}D_f(P'\|Q')\notag\\
 &\le(2c)^{j_{d_0}}
 \Df(\mu_1\star\cdots\star\mu_n).
 \label{eq:mf-large-defect-charge}
\end{align}
Insert \eqref{eq:mf-large-defect-charge} into
\eqref{eq:additive-RC}. This proves \eqref{eq:CLD-general}.
For $d_0=1/4$,
\[
 \frac{s_*}{s_{1/4}}<8,
\]
so $j_{1/4}\le3$, proving \eqref{eq:CLD-quarter}.
\end{proof}

It remains to treat
\begin{equation}
 \Delta:=\sum_i d(\mu_i)\le\frac14.
 \label{eq:small-defect}
\end{equation}

\section{Thinning of centered sums}\label{app:centered}

We prove Theorem~\ref{thm:centered-erasure} through a uniform comparison
for shifted positive parts. Conditioning each summand on its threshold
exceedance reduces this comparison to centered Bernoulli variables.

\subsection{A Bernoulli tail comparison}

Let $A_1,\ldots,A_n$ be independent Bernoulli variables with
\[
 \Pp(A_i=1)=p_i,
 \qquad 0\le p_i<1,
\]
and put
\[
 \theta_i:=\frac{p_i}{1-p_i},
 \qquad
 V_i:=
 \begin{cases}
  1,&A_i=1,\\
  -\theta_i,&A_i=0.
 \end{cases}
\]
Then $\E V_i=0$. Write
\[
 N:=\sum_iA_i,
 \qquad
 H:=\sum_iV_i.
\]

\begin{lemma}[Bernoulli tail comparison]
\label{lem:PB-score}
One has
\begin{equation}
 \Pp(N\ge2)
 \le
 7\,\E(H-1)_+.
 \label{eq:PB-score}
\end{equation}
\end{lemma}

\begin{proof}
Discard indices with $p_i=0$. The assertion is immediate if at most one
index remains. Otherwise, $n\ge2$ and every $\theta_i>0$.

Put
\[
 Z:=\prod_i(1+\theta_i),
\]
and let $e_k$ be the $k$th elementary symmetric polynomial in
$\theta_1,\ldots,\theta_n$, with $e_{n+1}:=0$. An active set
$B\subseteq[n]$ has probability
\[
 \Pp(\{i:A_i=1\}=B)
 =
 Z^{-1}\prod_{i\in B}\theta_i.
\]
If $|B|=k$, then
\[
 H(B)=k-\sum_{j\notin B}\theta_j.
\]
Consequently,
\begin{align}
 \sum_{|B|=k}
 \left(\prod_{i\in B}\theta_i\right)
 \bigl(H(B)-1\bigr)
 &=(k-1)e_k-(k+1)e_{k+1}\notag\\
 &=:a_k.\label{eq:ak}
\end{align}
Taking positive parts yields
\begin{equation}
 Z\,\E(H-1)_+
 \ge
 \sum_{k=2}^n(a_k)_+.\label{eq:PB-positive}
\end{equation}

Define
\[
 r_k:=\frac{(k+1)e_{k+1}}{e_k}.
\]
Newton's inequalities
\cite[Chapter~II, $\mathsection$2.22, Theorem~51]{HardyLittlewoodPolya} give
\[
 \frac{e_{k+1}e_{k-1}}{e_k^2}
 \le
 \frac{k}{k+1}\frac{n-k}{n-k+1}.
\]
Hence
\begin{equation}
 \frac{r_k}{r_{k-1}}
 =
 \frac{k+1}{k}\frac{e_{k+1}e_{k-1}}{e_k^2}
 \le
 \frac{n-k}{n-k+1}
 <1.\label{eq:r-decreasing}
\end{equation}
Thus $(r_k)$ is decreasing.

Let $m\ge2$ be the first index such that $a_m\ge0$, equivalently
$r_m\le m-1$. Such an index exists because $a_n=(n-1)e_n\ge0$.
By \eqref{eq:r-decreasing}, $a_k\ge0$ for every $k\ge m$. Telescoping
therefore gives
\begin{align}
 \sum_{k=m}^na_k
 &=
 \sum_{k=m}^n\bigl((k-1)e_k-(k+1)e_{k+1}\bigr)\notag\\
 &=
 (m-1)e_m-\sum_{k=m+1}^ne_k.\label{eq:a-telescope}
\end{align}
For $k\ge m$,
\[
 \frac{e_{k+1}}{e_k}
 =\frac{r_k}{k+1}
 \le\frac{m-1}{k+1}
 \le\frac{m-1}{m+1}.
\]
More precisely, iteration gives
\[
 e_{m+j}
 \le
 e_m\frac{(m-1)^j}{(m+1)(m+2)\cdots(m+j)},
\]
and summing the geometric bound with ratio $(m-1)/(m+1)$ yields
\[
 \sum_{k=m+1}^ne_k\le\frac{m-1}{2}e_m.
\]
Together with \eqref{eq:a-telescope},
\begin{equation}
 \sum_{k=2}^n(a_k)_+
 \ge
 \frac{m-1}{2}e_m.\label{eq:a-lower}
\end{equation}

Suppose first that $m\ge3$. For $j<m$, one has $a_j<0$, hence
$r_j>j-1$. Since $(r_j)$ is decreasing,
\[
 r_j\ge r_{m-1}>m-2
 \qquad(j<m).
\]
Iterating $e_{j+1}=r_je_j/(j+1)$ backward, for $2\le k<m$ we obtain
\begin{equation}
 e_k
 \le
 e_m\frac{m!}{k!(m-2)^{m-k}}.\label{eq:e-backward}
\end{equation}
The coefficients in \eqref{eq:e-backward} are at most $3$: for $m=3$
the sole coefficient is $3$; for $m\ge4$, the largest coefficient with
$k\le m-2$ occurs at $k=m-2$ and equals
\[
 \frac{m(m-1)}{(m-2)^2}\le3,
\]
whereas the coefficient at $k=m-1$ is $m/(m-2)\le2$. Consequently,
\[
 \sum_{k=2}^me_k\le3(m-1)e_m.
\]
The preceding tail estimate now gives
\begin{equation}
 \sum_{k=2}^ne_k
 \le
 \frac72(m-1)e_m.\label{eq:e-total}
\end{equation}
Combining \eqref{eq:a-lower} and \eqref{eq:e-total},
\begin{equation}
 \sum_{k=2}^ne_k
 \le
 7\sum_{k=2}^n(a_k)_+.\label{eq:e-a}
\end{equation}

If $m=2$, then $r_k\le1$ for every $k\ge2$, and therefore
$e_{k+1}/e_k\le1/(k+1)$. It follows that
\[
 \sum_{k=3}^ne_k<\frac12e_2.
\]
On the other hand,
\[
 \sum_{k=2}^n(a_k)_+
 =e_2-\sum_{k=3}^ne_k
 >\frac12e_2,
\]
so \eqref{eq:e-a} also holds.

Finally,
\[
 \Pp(N\ge2)
 =\frac1Z\sum_{k=2}^ne_k
 \le
 \frac7Z\sum_{k=2}^n(a_k)_+
 \le
 7\,\E(H-1)_+,
\]
by \eqref{eq:PB-positive}.
\end{proof}

\subsection{Shifted Bernoulli sums}


\begin{lemma}[Shift comparison]
\label{lem:Boolean-refund}
Let $V_i$ be as above, let $c_i\ge0$, and put
\[
 C:=\sum_ic_iV_i,
 \qquad
 H:=\sum_iV_i,
 \qquad
 S:=C+H.
\]
Then
\begin{equation}
 \E C_+
 \le
 8\,\E(S-1)_+.
 \label{eq:Boolean-refund}
\end{equation}
Consequently, for every $t>0$,
\begin{equation}
 \E\left(\sum_ic_iV_i\right)_+
 \le
 8\,\E\left(\sum_i(c_i+t)V_i-t\right)_+.
 \label{eq:Boolean-refund-scaled}
\end{equation}
\end{lemma}

\begin{proof}
Identify a realization with its active set $B\subseteq[n]$. Changing
coordinate $i$ from inactive to active increases $C$ by
$(1+\theta_i)c_i\ge0$. Hence
\[
 \mathcal U:=\{C>0\}
\]
is increasing, and it does not contain the empty set. Put
$z:=\one_{\mathcal U}$. For a function $g$ on the Boolean lattice,
define
\[
 D_ig(B)
 :=
 \one_{\{i\in B\}}
 \bigl(g(B)-g(B\setminus\{i\})\bigr).
\]
Conditioning on the other coordinates gives
\begin{equation}
 \E[V_ig]=\E[D_ig].\label{eq:Boolean-IBP}
\end{equation}
Both conditional expectations equal
\[
 p_i\bigl(g(B\cup\{i\})-g(B\setminus\{i\})\bigr).
\]

Let
\[
 k_{\mathcal U}:=\sum_iD_iz
\]
be the number of active pivotal coordinates. By
\eqref{eq:Boolean-IBP},
\begin{equation}
 \E[(1-H)z]
 =
 \E[z-k_{\mathcal U}].\label{eq:pivotal-identity}
\end{equation}
Every singleton in $\mathcal U$ has a pivotal coordinate, since
$\varnothing\notin\mathcal U$. Thus
\[
 z-k_{\mathcal U}\le\one_{\{N\ge2\}}.
\]
Lemma~\ref{lem:PB-score} yields
\begin{equation}
 \E[(1-H)z]
 \le
 7\,\E(H-1)_+.\label{eq:refund-error}
\end{equation}

We also have
\begin{equation}
 H\cx S.\label{eq:H-S-cx}
\end{equation}
If $X$ is centered and $\lambda\ge1$, then $X\cx\lambda X$. Indeed,
for every convex $\varphi$, the function
$r\mapsto\E\varphi(rX)$ is convex and is minimized at $r=0$ by
Jensen's inequality; it is therefore nondecreasing on $[0,\infty)$.
Apply this to $V_i$ with $\lambda=1+c_i$, and use
Lemma~\ref{lem:cx-convolution}. It follows that
\begin{equation}
 \E(H-1)_+
 \le
 \E(S-1)_+.\label{eq:H-S-stoploss}
\end{equation}
Finally,
\[
 \begin{aligned}
 \E(S-1)_+
 &\ge \E[(S-1)z]\\
 &=\E[Cz]-\E[(1-H)z]\\
 &=\E C_+-\E[(1-H)z].
 \end{aligned}
\]
Using \eqref{eq:refund-error} and \eqref{eq:H-S-stoploss},
\[
 \E(S-1)_+
 \ge
 \E C_+-7\E(S-1)_+,
\]
which proves \eqref{eq:Boolean-refund}. Apply it with $c_i/t$ and
multiply by $t$ to obtain \eqref{eq:Boolean-refund-scaled}.
\end{proof}

\subsection{Shifted positive parts}

\begin{lemma}[Rademacher lower bound]
\label{lem:Rademacher-lower}
For arbitrary real numbers $a_1,\ldots,a_n$ and independent Rademacher
variables $\varepsilon_i$,
\begin{equation}
 \E\left|\sum_i\varepsilon_i a_i\right|
 \ge
 \frac1{\sqrt3}
 \left(\sum_i a_i^2\right)^{1/2}.\label{eq:Rademacher-lower}
\end{equation}
\end{lemma}

\begin{proof}
Put $R=\sum_i\varepsilon_i a_i$ and
$\sigma^2=\sum_i a_i^2$. Then
\[
 \E R^2=\sigma^2,
 \qquad
 \E R^4
 =\sum_i a_i^4+6\sum_{i<j}a_i^2a_j^2
 \le3\sigma^4.
\]
H\"older's inequality gives
\[
 \E R^2
 \le
 (\E|R|)^{2/3}(\E|R|^4)^{1/3}.
\]
Consequently,
\[
 \E|R|
 \ge
 \frac{(\E R^2)^{3/2}}{(\E R^4)^{1/2}}
 \ge
 \frac\sigma{\sqrt3}.
\]
\end{proof}

\begin{lemma}[$L^1$ comparison]
\label{lem:L1-cotype-zero}
Let $Z_1,\ldots,Z_n$ be independent centered integrable real variables.
Then
\begin{equation}
 \left(\sum_i[\E(Z_i)_+]^2\right)^{1/2}
 \le
 2\sqrt3\,
 \E\left(\sum_iZ_i\right)_+.\label{eq:L1-cotype-zero}
\end{equation}
\end{lemma}

\begin{proof}
Let $Z_i'$ be independent copies and put $D_i=Z_i-Z_i'$. The $D_i$
are independent and symmetric. Conditional on their absolute values,
Lemma~\ref{lem:Rademacher-lower} gives
\[
 \E\left|\sum_iD_i\right|
 \ge
 \frac1{\sqrt3}
 \E\left(\sum_i|D_i|^2\right)^{1/2}.
\]
The Euclidean norm is convex, so
\[
 \E\left(\sum_i|D_i|^2\right)^{1/2}
 \ge
 \left(\sum_i[\E|D_i|]^2\right)^{1/2}.
\]
Moreover, conditional Jensen gives $\E|D_i|\ge\E|Z_i|$, whereas
\[
 \E\left|\sum_iD_i\right|
 \le
 2\E\left|\sum_iZ_i\right|.
\]
Because every variable involved is centered,
\[
 \E(Z_i)_+=\frac12\E|Z_i|,
 \qquad
 \E\left(\sum_iZ_i\right)_+
 =\frac12\E\left|\sum_iZ_i\right|.
\]
Combining these inequalities proves \eqref{eq:L1-cotype-zero}.
\end{proof}

\begin{theorem}[Shifted positive parts]
\label{thm:shifted-cotype}
Let $Z_1,\ldots,Z_n$ be independent centered integrable real variables.
Then, for every $t\ge0$,
\begin{equation}
 \left(\sum_i[\E(Z_i-t)_+]^2\right)^{1/2}
 \le
 16\sqrt3\,
 \E\left(\sum_iZ_i-t\right)_+.
 \label{eq:shifted-cotype}
\end{equation}
\end{theorem}

\begin{proof}
The case $t=0$ is Lemma~\ref{lem:L1-cotype-zero}. Let $t>0$ and put
$p_i=\Pp(Z_i>t)$.

Discard indices with $p_i=0$. If $X$ and $W$ denote the retained and
discarded sums, respectively, then $W$ is centered and independent of
$X$, so conditional Jensen gives
\[
 X\cx X+W
\]
and hence
\[
 \E(X-t)_+\le\E(X+W-t)_+.
\]
The discarded indices contribute zero to the left-hand side, so it
suffices to prove the result for the retained sum. We may therefore
assume $0<p_i<1$, the upper bound following from centeredness. Let
\[
 \mathcal H
 :=
 \sigma(\one_{\{Z_i>t\}}:1\le i\le n),
 \qquad
 Y_i:=\E[Z_i\mid\one_{\{Z_i>t\}}].
\]
Each $Y_i$ is centered and two-point, so for some $c_i>0$,
\[
 Y_i=(c_i+t)V_i,
\]
where $V_i$ is as in Lemma~\ref{lem:Boolean-refund}. On $\{Z_i>t\}$,
$Y_i=c_i+t$; on the complement,
\[
 Y_i=-\frac{p_i}{1-p_i}(c_i+t).
\]
Furthermore,
\begin{equation}
 \E(Z_i-t)_+
 =p_ic_i
 =\E(c_iV_i)_+.\label{eq:individual-stoploss}
\end{equation}
Lemma~\ref{lem:L1-cotype-zero}, followed by
\eqref{eq:Boolean-refund-scaled}, gives
\[
 \begin{aligned}
 \left(\sum_i[\E(Z_i-t)_+]^2\right)^{1/2}
 &\le2\sqrt3\,
 \E\left(\sum_ic_iV_i\right)_+\\
 &\le16\sqrt3\,
 \E\left(\sum_i(c_i+t)V_i-t\right)_+\\
 &=16\sqrt3\,
 \E\left(\sum_iY_i-t\right)_+.
 \end{aligned}
\]
Since
\[
 \sum_iY_i
 =
 \E\left[\sum_iZ_i\mid\mathcal H\right],
\]
conditional Jensen finishes the proof.
\end{proof}

\subsection{Proof of the convex thinning inequality}

Theorem~\ref{thm:shifted-cotype} implies
Theorem~\ref{thm:centered-erasure} by random partitioning and the
positive-part representation of convex functions.


\begin{proof}[Proof of Theorem~\ref{thm:centered-erasure}]
First suppose $q=1/m$, where $m\ge1$ is an integer. Independently
color every index uniformly with one of $m$ colors and put
\[
 W_j:=\sum_{\{i:C_i=j\}}Z_i,
 \qquad 1\le j\le m.
\]
Conditional on the coloring, $W_1,\ldots,W_m$ are independent and
centered, with $\sum_jW_j=\sum_iZ_i$. Theorem~\ref{thm:shifted-cotype}
gives
\[
 \left(\sum_{j=1}^m[\E(W_j-t)_+]^2\right)^{1/2}
 \le
 16\sqrt3\,
 \E\left(\sum_iZ_i-t\right)_+.
\]
By Cauchy--Schwarz,
\[
 \frac1m\sum_{j=1}^m\E(W_j-t)_+
 \le
 \frac{16\sqrt3}{\sqrt m}
 \E\left(\sum_iZ_i-t\right)_+.
\]
For an independent uniform color $J$, the indicators
$\one_{\{C_i=J\}}$ are conditionally independent $\Ber(1/m)$ variables
whose joint law does not depend on $J$. Thus $W_J$ has the law of the
$1/m$-thinned sum, and averaging over the coloring gives
\begin{equation}
 \E\left(\sum_iB_iZ_i-t\right)_+
 \le
 16\sqrt{3q}\,
 \E\left(\sum_iZ_i-t\right)_+\label{eq:erasure-hinge-recip}
\end{equation}
when $q=1/m$.

Now let $0<q<1/2$ and put $m=\lfloor1/q\rfloor$. Then
$q\le1/m\le2q$. If $Z$ is centered and $0\le r\le s\le1$, then
$B_rZ\cx B_sZ$, because for every convex $\varphi$,
\[
 \E\varphi(B_rZ)
 =(1-r)\varphi(0)+r\E\varphi(Z)
\]
is nondecreasing in $r$ by Jensen's inequality. Applying
Lemma~\ref{lem:cx-convolution} coordinatewise and then
\eqref{eq:erasure-hinge-recip} at $1/m$ gives
\begin{equation}
 \E\left(\sum_iB_iZ_i-t\right)_+
 \le
 16\sqrt{6q}\,
 \E\left(\sum_iZ_i-t\right)_+
 \le
 40\sqrt q\,
 \E\left(\sum_iZ_i-t\right)_+.\label{eq:erasure-positive-hinge}
\end{equation}
For $q\ge1/2$, convex-order comparison with the unthinned sum gives
\eqref{eq:erasure-positive-hinge}; $q=0$ is immediate.
Apply the same inequality to $-Z_i$.

The finite convex function $\psi$ is locally absolutely continuous,
with finite one-sided derivatives. Put
\[
 a_+:=\psi'_+(0),
 \qquad
 a_-:=-\psi'_-(0).
\]
Because zero is a minimizer, $a_\pm\ge0$. Let $\nu_+$ be the
Lebesgue--Stieltjes measure generated on $(0,\infty)$ by the
nondecreasing right derivative of $\psi$, and define $\nu_-$ in the same
way from the convex function $x\mapsto\psi(-x)$. For $x\ge0$,
\[
 \begin{aligned}
 \psi(x)
 &=\int_0^x\psi'_+(u)\,\dd u\\
 &=a_+x+\int_{(0,x]}(x-t)\,\nu_+(\dd t),
 \end{aligned}
\]
where the second equality follows by inserting
$\psi'_+(u)=a_++\nu_+((0,u])$ and applying Tonelli. The reflected formula
holds on the negative half-line. Hence, for every $x\in\R$,
\[
 \begin{aligned}
 \psi(x)
 ={}&a_+x_++a_-x_-\\
 &+\int_{(0,\infty)}(x-t)_+\,\nu_+(\dd t)
 +\int_{(0,\infty)}(-x-t)_+\,\nu_-(\dd t).
 \end{aligned}
\]
The measures $\nu_\pm$ are positive and locally finite. Integrating
\eqref{eq:erasure-positive-hinge} and its reflected version against this
representation proves the result by Tonelli's theorem.
\end{proof}

\section{Thinning at small squared Hellinger distance}\label{app:nearby}
\label{app:small-distance}
\subsection{Completion and translation bounds}

We transfer Theorem~\ref{thm:centered-erasure} to Hellinger measures by
completion at zero. The completed variables satisfy $\E\sinh X_i=0$;
their ordinary means and the resulting translations are controlled by
the total squared Hellinger distance.


Let $\eta_1,\ldots,\eta_n\in\mathcal A$, and put
\[
 d_i:=1-\eta_i(\R),
 \qquad
 \Delta:=\sum_i d_i.
\]
Assume throughout that $\Delta\le1/4$, and define the probability laws
\begin{equation}
 \widehat\eta_i:=\eta_i+d_i\delta_0.\label{eq:completion}
\end{equation}
Then
\[
 \widehat\eta_i(\R)=1,
 \qquad
 \int e^{\pm x}\,\widehat\eta_i(\dd x)=1+d_i.
\]
Let $X_i\sim\widehat\eta_i$ independently and put
\[
 S:=\sum_iX_i.
\]
Since
\[
 \int\sinh x\,\widehat\eta_i(\dd x)=0,
\]
we have
\begin{equation}
 \E\sinh X_i=0.\label{eq:sinh-centering}
\end{equation}
Moreover,
\[
 \E(\cosh X_i-1)=d_i.
\]
The elementary inequality
\[
 |\sinh x-x|\le\cosh x-1
\]
therefore gives
\begin{equation}
 m_i:=\E X_i,
 \qquad
 |m_i|\le d_i.\label{eq:mean-bound}
\end{equation}

Put
\[
 \eta_*:=\eta_1*\cdots*\eta_n,
 \qquad
 \widehat Q:=\widehat\eta_1*\cdots*\widehat\eta_n,
\]
and
\begin{equation}
 A_\Delta:=\prod_i(1+d_i),
 \qquad
 \overline\eta:=A_\Delta^{-1}\widehat Q.\label{eq:normalized-completion}
\end{equation}
Then $\overline\eta\in\mathcal A$.

\begin{lemma}[Divergence after completion]
\label{lem:completion-energy}
One has
\begin{equation}
 \Ff(\overline\eta)\le\Ff(\eta_*).\label{eq:completion-energy}
\end{equation}
\end{lemma}

\begin{proof}
For $A\subseteq[n]$, put
\[
 \eta_A:=\mathop{*}_{i\in A}\eta_i,
 \qquad
 \eta_\varnothing:=\delta_0.
\]
Expanding the completed convolution,
\[
 \widehat Q
 =
 \sum_{A\subseteq[n]}
 \left(\prod_{i\notin A}d_i\right)\eta_A.
\]
Since $\eta_A\in\mathcal A$ is a convolution factor of $\eta_*$,
convolution monotonicity gives
\[
 \Ff(\eta_A)\le\Ff(\eta_*).
\]
Hence
\[
 \Ff(\widehat Q)
 \le
 \left[\sum_{A\subseteq[n]}\prod_{i\notin A}d_i\right]
 \Ff(\eta_*)
 =A_\Delta\Ff(\eta_*).
\]
Divide by $A_\Delta$.
\end{proof}

\begin{proposition}[Bounds under small translations]
\label{prop:endpoint-reserve}
For every $|s|\le\Delta$,
\begin{equation}
 \E h(S+s)
 \le
 C_{\mathrm{end}}(c)\,\Ff(\eta_*),
 \label{eq:endpoint-reserve}
\end{equation}
where one may take
\begin{equation}
 C_{\mathrm{end}}(c)
 :=2e^{1/4}c\cosh(1/4)(2c^2+1).
 \label{eq:Cend-explicit}
\end{equation}
\end{proposition}

\begin{proof}
If $\Delta=0$, then
\[
 0=d_i=\int(\cosh x-1)\,\eta_i(\dd x)
\]
for every $i$, whence $\eta_i=\delta_0$ and the assertion follows.
Assume $\Delta>0$.

The total mass of $\overline\eta$ is $A_\Delta^{-1}$, so its squared
Hellinger distance is
\[
 \delta:=1-\overline\eta(\R)=1-\frac1{A_\Delta}.
\]
Because
\[
 A_\Delta=\prod_i(1+d_i)\ge1+\sum_i d_i=1+\Delta,
\]
we have
\begin{equation}
 \delta\ge\frac{\Delta}{1+\Delta}.\label{eq:normalized-defect}
\end{equation}

Let $U$ be the centered posterior score associated with
$\overline\eta$, and put
\[
 r:=\E U_+=\E U_-=\frac12\E|U|.
\]
Since $|U|\ge U^2$, Lemma~\ref{lem:posterior-identities} and
\eqref{eq:normalized-defect} imply
\begin{equation}
 r
 \ge\frac12\E U^2
 \ge\frac\delta2
 \ge\frac{\Delta}{2(1+\Delta)}.\label{eq:r-lower}
\end{equation}
Let
\[
 p_+:=\Pp(U>0),
 \qquad
 u_+:=\E[U\mid U>0].
\]
Then $r=p_+u_+$. Convexity of $\Phi_f$ and $\Phi_f(0)=0$ give
\[
 \Phi_f(r)
 =\Phi_f(p_+u_++(1-p_+)0)
 \le p_+\Phi_f(u_+)
 \le\E[\Phi_f(U)\one_{\{U>0\}}].
\]
The reflected argument gives
\[
 \Phi_f(-r)
 \le\E[\Phi_f(U)\one_{\{U<0\}}].
\]
Consequently,
\begin{equation}
 \mathcal M_f(\operatorname{arctanh}r)
 =\Phi_f(r)+\Phi_f(-r)
 \le\Ff(\overline\eta).\label{eq:M-anchor}
\end{equation}

For $|s|\le\Delta$, \eqref{eq:r-lower} gives
\[
 \frac{|s|}{\operatorname{arctanh}r}
 \le\frac{\Delta}{r}
 \le2(1+\Delta)
 \le\frac52.
\]
Monotonicity and two applications of \eqref{eq:M-doubling} therefore give
\begin{equation}
 \mathcal M_f(|s|)
 \le4c^2\mathcal M_f(\operatorname{arctanh}r)
 \le4c^2\Ff(\eta_*),
 \label{eq:M-endpoint}
\end{equation}
where Lemma~\ref{lem:completion-energy} was used last.

Consider the symmetric two-point measure
\[
 \omega_s
 =
 \frac{\delta_s+\delta_{-s}}{2\cosh s}.
\]
Equation \eqref{eq:M-endpoint} gives
\[
 \Ff(\omega_s)
 =\frac12\mathcal M_f(|s|)
 \le2c^2\Ff(\eta_*).
\]
Furthermore,
\[
 \omega_s*\overline\eta
 =
 \frac{\delta_s*\overline\eta+\delta_{-s}*\overline\eta}
      {2\cosh s}.
\]
Since $h\ge0$, the product comparison
\eqref{eq:reverse-product} gives
\begin{align*}
 \Ff(\delta_s*\overline\eta)
 &\le2\cosh s\,\Ff(\omega_s*\overline\eta)\\
 &\le2c\cosh(1/4)
   [\Ff(\omega_s)+\Ff(\overline\eta)]\\
 &\le2c\cosh(1/4)(2c^2+1)\Ff(\eta_*).
\end{align*}
Finally, $\widehat Q=A_\Delta\overline\eta$ and
$A_\Delta\le e^{1/4}$, so
\[
 \E h(S+s)
 =A_\Delta\Ff(\delta_s*\overline\eta)
 \le C_{\mathrm{end}}(c)\Ff(\eta_*).
\]
\end{proof}

\subsection{Centering the thinned sum}

The following estimate controls the random translation introduced by
centering, with a factor of $q$.


\begin{lemma}[Random translation]
\label{lem:random-drift}
Let $Z_1,\ldots,Z_n$ be independent centered real variables, let
$m_1,\ldots,m_n\in\R$, and let $B_i\sim\Ber(q)$ be independent of
one another and of the $Z_i$. Put
\[
 Z:=\sum_iZ_i,
 \qquad
 A:=\sum_iB_iZ_i,
 \qquad
 R:=\sum_iB_im_i,
\]
and
\[
 M_+:=\sum_i(m_i)_+,
 \qquad
 M_-:=\sum_i(-m_i)_+.
\]
For every nonnegative convex function $h$,
\begin{equation}
 \E h(A+R)
 \le
 \E h(A)
 +q\,\E h(Z+M_+)
 +q\,\E h(Z-M_-).
 \label{eq:random-drift}
\end{equation}
\end{lemma}

\begin{proof}
For every realization, $-M_-\le R\le M_+$. Define
\[
 r_+:=
 \begin{cases}
  R_+/M_+,&M_+>0,\\
  0,&M_+=0,
 \end{cases}
 \qquad
 r_-:=
 \begin{cases}
  (-R)_+/M_-,&M_->0,\\
  0,&M_-=0.
 \end{cases}
\]
At most one of $r_+,r_-$ is nonzero, $0\le r_\pm\le1$, and
\[
 A+R
 =(1-r_+-r_-)A+r_+(A+M_+)+r_-(A-M_-).
\]
Convexity and nonnegativity imply
\begin{equation}
 h(A+R)
 \le
 h(A)+r_+h(A+M_+)+r_-h(A-M_-).\label{eq:drift-pathwise}
\end{equation}

If $M_+>0$, then
\[
 r_+
 \le
 \sum_{i:m_i>0}\frac{m_i}{M_+}B_i.
\]
Hence
\[
 \E[r_+h(A+M_+)]
 \le
 \sum_{i:m_i>0}\frac{m_i}{M_+}
 \E[B_ih(A+M_+)].
\]
Conditioning on $B_i=1$,
\[
 \E[B_ih(A+M_+)]
 =q\,\E h\left(M_++Z_i+\sum_{j\ne i}B_jZ_j\right).
\]
Since $B_jZ_j\cx Z_j$, Lemma~\ref{lem:cx-convolution} gives
\[
 Z_i+\sum_{j\ne i}B_jZ_j\cx\sum_jZ_j=Z.
\]
Applying the convex function $x\mapsto h(M_++x)$ gives
\[
 \E[B_ih(A+M_+)]
 \le q\,\E h(Z+M_+).
\]
The weights $m_i/M_+$ sum to one, so
\[
 \E[r_+h(A+M_+)]
 \le q\,\E h(Z+M_+).
\]
The negative part is identical:
\[
 \E[r_-h(A-M_-)]
 \le q\,\E h(Z-M_-).
\]
Insert these estimates in \eqref{eq:drift-pathwise}.
\end{proof}

\subsection{Thinning Hellinger measures}

\begin{theorem}[Thinning at small squared Hellinger distance]
\label{thm:Hellinger-erasure}
Let $\eta_1,\ldots,\eta_n\in\mathcal A$, and assume
\[
 \Delta:=\sum_i[1-\eta_i(\R)]\le\frac14.
\]
Then, for every $q\in[0,1]$,
\begin{equation}
 \Ff\!\left(
  \mathop{*}_{i=1}^n[(1-q)\delta_0+q\eta_i]
 \right)
 \le
 L(c)\sqrt q\,
 \Ff(\eta_1*\cdots*\eta_n).
 \label{eq:Hellinger-erasure}
\end{equation}
\end{theorem}

\begin{proof}
Complete the factors as in \eqref{eq:completion}, let
$X_i\sim\widehat\eta_i$ independently, and put
\[
 m_i:=\E X_i,
 \qquad
 Z_i:=X_i-m_i.
\]
Then $\E Z_i=0$ and, by \eqref{eq:mean-bound},
\[
 \sum_i|m_i|\le\Delta.
\]
Set
\[
 M:=\sum_i m_i,
 \qquad
 M_+:=\sum_i(m_i)_+,
 \qquad
 M_-:=\sum_i(-m_i)_+.
\]
Thus $|M|,M_+,M_-\le\Delta$. Let $B_i\sim\Ber(q)$ independently and
write
\[
 A_B:=\sum_iB_iZ_i,
 \qquad
 R_B:=\sum_iB_im_i.
\]
By Theorem~\ref{thm:centered-erasure},
\[
 \E h(A_B)
 \le
 40\sqrt q\,
 \E h\left(\sum_iZ_i\right).
\]
Since $\sum_iZ_i=S-M$ and $|M|\le\Delta$,
Proposition~\ref{prop:endpoint-reserve} gives
\begin{equation}
 \E h(A_B)
 \le
 40C_{\mathrm{end}}(c)\sqrt q\,\Ff(\eta_*),
 \qquad
 \eta_*:=\eta_1*\cdots*\eta_n.\label{eq:centered-part-bound}
\end{equation}

Apply Lemma~\ref{lem:random-drift}:
\[
 \E h(A_B+R_B)
 \le
 \E h(A_B)
 +q\,\E h(Z+M_+)
 +q\,\E h(Z-M_-),
\]
where $Z=\sum_iZ_i=S-M$. Since $M=M_+-M_-$,
\[
 Z+M_+=S+M_-,
 \qquad
 Z-M_-=S-M_+.
\]
Both translations have absolute value at most $\Delta$, and hence
Proposition~\ref{prop:endpoint-reserve} yields
\[
 \E h(Z+M_+)+\E h(Z-M_-)
 \le
 2C_{\mathrm{end}}(c)\Ff(\eta_*).
\]
Together with \eqref{eq:centered-part-bound} and $q\le\sqrt q$,
\begin{equation}
 \E h\left(\sum_iB_iX_i\right)
 \le
 42C_{\mathrm{end}}(c)\sqrt q\,\Ff(\eta_*).\label{eq:completed-erasure}
\end{equation}

The law of $\sum_iB_iX_i$ is
\[
 \widehat\eta^{[q]}
 :=
 \mathop{*}_{i=1}^n[(1-q)\delta_0+q\widehat\eta_i].
\]
For each $i$,
\[
 (1-q)\delta_0+q\widehat\eta_i
 =(1-q+qd_i)\delta_0+q\eta_i
 \ge
 (1-q)\delta_0+q\eta_i
\]
as positive measures. Convolution preserves positive-measure
domination. Since $h\ge0$,
\[
 \Ff\!\left(
  \mathop{*}_{i=1}^n[(1-q)\delta_0+q\eta_i]
 \right)
 \le
 \Ff(\widehat\eta^{[q]})
 =
 \E h\left(\sum_iB_iX_i\right).
\]
Use \eqref{eq:completed-erasure} and take
$L(c)=42C_{\mathrm{end}}(c)$.
\end{proof}

Affineness of the Hellinger correspondence and
$\eta_{\delta_1}=\delta_0$ give the following consequence.

\begin{corollary}[Thinning likelihood ratios]
\label{cor:likelihood-erasure}
Let $\mu_1,\ldots,\mu_n$ be positive mean-one laws satisfying
$\sum_i d(\mu_i)\le1/4$. Then, for every $q\in[0,1]$,
\begin{equation}
 \Df\!\left(
  \Astar_{i=1}^n[(1-q)\delta_1+q\mu_i]
 \right)
 \le
 L(c)\sqrt q\,
 \Df(\mu_1\star\cdots\star\mu_n).
 \label{eq:likelihood-erasure}
\end{equation}
\end{corollary}

\section{Poissonization and homogenization}\label{app:poisson}

Poissonization converts the sampling multiplicities into independent
$\Pois(1)$ variables. Their level sets are Bernoulli thinnings with
parameters $q_k=\Pp\{\Pois(1)\ge k\}$. We combine these levels along a
binary tree whose depths are chosen by Kraft's inequality; the
factorial decay of $q_k$ controls the accumulated comparison constants.


For a mean-one law $\rho$ and $s\ge0$, define the compound-Poisson law
\[
 \Pi_s(\rho)
 :=
 e^{-s}\sum_{r=0}^\infty\frac{s^r}{r!}\rho^{\star r}.
\]

\begin{lemma}[Poisson median bound]
\label{lem:Poisson-median}
If $N\sim\Pois(n)$ with integer $n\ge1$, then
\[
 \Pp(N\ge n)\ge\frac12.
\]
Consequently,
\begin{equation}
 \Df(\rho^{\star n})
 \le
 2\Df(\Pi_n(\rho)).\label{eq:depoisson-lower}
\end{equation}
\end{lemma}

\begin{proof}
Write $p_k=e^{-n}n^k/k!$. For $j=1,\ldots,n$,
\[
 \frac{p_{n+j-1}}{p_{n-j}}
 =
 \frac{n^{2j-1}}
 {(n-j+1)(n-j+2)\cdots(n+j-1)}.
\]
The denominator equals
\[
 n\prod_{\ell=1}^{j-1}(n^2-\ell^2)
 \le n^{2j-1}.
\]
Thus $p_{n-j}\le p_{n+j-1}$. Pairing the masses below $n$ with those
from $n$ through $2n-1$ proves the median estimate.

By linearity of $\Df$ in its law,
\[
 \Df(\Pi_n(\rho))=\E\Df(\rho^{\star N}).
\]
On $\{N\ge n\}$, convolution monotonicity gives
$\Df(\rho^{\star N})\ge\Df(\rho^{\star n})$. The conclusion follows.
\end{proof}

\begin{theorem}[From thinning to sampling with replacement]
\label{thm:power-erasure}
Let $\mu_1,\ldots,\mu_n$ be a mean-one family, and define
\[
 F_{\boldsymbol\mu}(q)
 :=
 \Df\!\left(
  \Astar_{i=1}^n[(1-q)\delta_1+q\mu_i]
 \right).
\]
Suppose that, for some $L<\infty$ and $\beta>0$,
\begin{equation}
 F_{\boldsymbol\mu}(q)
 \le
 Lq^\beta F_{\boldsymbol\mu}(1),
 \qquad0\le q\le1.
 \label{eq:power-erasure-assumption}
\end{equation}
For
\[
 q_k:=\Pp\{\Pois(1)\ge k\},
 \qquad
 S_r:=\sum_{k\ge1}q_k^r,
\]
and every $0<\gamma<\beta/a$, one has
\begin{equation}
 \Df(\overline\mu^{\star n})
 \le
 2Lc^2S_\gamma^aS_{\beta-a\gamma}\,
 \Df(\mu_1\star\cdots\star\mu_n),
 \label{eq:power-erasure-conclusion}
\end{equation}
where $\overline\mu=n^{-1}\sum_i\mu_i$.
\end{theorem}

\begin{proof}
Let $N_1,\ldots,N_n$ be independent $\Pois(1)$ variables. For $k\ge1$,
put
\[
 B_{i,k}:=\one_{\{N_i\ge k\}},
 \qquad
 q_k:=\Pp\{\Pois(1)\ge k\},
\]
and define the random mean-one law
\[
 L_k:=\Astar_{i=1}^n\mu_i^{\star B_{i,k}}.
\]
Pathwise,
\begin{equation}
 \Astar_{i=1}^n\mu_i^{\star N_i}
 =
 \Astar_{k\ge1}L_k.\label{eq:Poisson-layers}
\end{equation}
Only finitely many factors on the right differ from $\delta_1$.
Compound-Poisson superposition gives
\begin{equation}
 \Astar_{i=1}^n\Pi_1(\mu_i)
 =
 \Pi_n(\overline\mu).\label{eq:Poisson-superposition}
\end{equation}
Indeed, the multinomial formula gives
\begin{align*}
 \Astar_{i=1}^n\Pi_1(\mu_i)
 &=
 e^{-n}
 \sum_{r_1,\ldots,r_n\ge0}
 \frac{\mu_1^{\star r_1}\star\cdots\star\mu_n^{\star r_n}}
      {r_1!\cdots r_n!}\\
 &=
 e^{-n}\sum_{r=0}^\infty
 \frac1{r!}\left(\sum_{i=1}^n\mu_i\right)^{\star r}\\
 &=
 e^{-n}\sum_{r=0}^\infty
 \frac{n^r}{r!}\,\overline\mu^{\star r}
 =\Pi_n(\overline\mu).
\end{align*}
All rearrangements are justified coefficientwise for positive measures.

Choose
\[
 0<\gamma<\frac\beta a,
 \qquad a=\log_2c.
\]
Since
\[
 q_k
 =\Pp\{\Pois(1)\ge k\}
 \le\frac{\E(N_1)_k}{k!}
 =\frac1{k!},
\]
both series
\[
 \sum_{k\ge1}q_k^\gamma,
 \qquad
 \sum_{k\ge1}q_k^{\beta-a\gamma}
\]
converge. Put $S_\gamma:=\sum_{k\ge1}q_k^\gamma$ and choose
$L_0:=\lceil\log_2S_\gamma\rceil$. Then
\[
 2^{-L_0}S_\gamma\le1.
\]
Set
\[
 \ell_k
 :=
 L_0+\left\lceil\gamma\log_2\frac1{q_k}\right\rceil.
\]
Then $\sum_k2^{-\ell_k}\le1$. The extended Kraft inequality
\cite[Theorem~5.2.2, pp.~109--110]{CoverThomas} supplies a countable
binary prefix code with lengths $\ell_k$, or equivalently a binary
tree with leaf depths $\ell_k$. Such a code can also be obtained from
disjoint dyadic subintervals of $[0,1)$ of lengths $2^{-\ell_k}$.

For each realization, retain the finite subtree connecting the
nontrivial leaves to the root, and insert $\delta_1$ at each missing
child. Iterating \eqref{eq:reverse-product} over this finite tree gives
\[
 \Df\!\left(\Astar_{k\ge1}L_k\right)
 \le
 \sum_{k\ge1}c^{\ell_k}\Df(L_k).
\]
Taking expectations and using linearity of $\Df$ under mixtures,
\begin{align}
 \Df(\Pi_n(\overline\mu))
 &=
 \E\Df\!\left(\Astar_i\mu_i^{\star N_i}\right)\notag\\
 &\le
 \sum_{k\ge1}c^{\ell_k}\E\Df(L_k).\label{eq:tree-bound}
\end{align}
For fixed $k$, the $B_{i,k}$ are independent $\Ber(q_k)$ variables, so
\[
 \E\Df(L_k)=F_{\boldsymbol\mu}(q_k).
\]
Moreover,
\[
 c^{\ell_k}\le c^{L_0+1}q_k^{-a\gamma}.
\]
Since
\[
 c^{L_0+1}
 \le c^2S_\gamma^{\log_2c}
 =c^2S_\gamma^a,
\]
using \eqref{eq:power-erasure-assumption} in \eqref{eq:tree-bound}
gives
\[
 \Df(\Pi_n(\overline\mu))
 \le
 Lc^2S_\gamma^aS_{\beta-a\gamma}
 F_{\boldsymbol\mu}(1).
\]
Finally, Lemma~\ref{lem:Poisson-median} and
$F_{\boldsymbol\mu}(1)=\Df(\mu_1\star\cdots\star\mu_n)$ prove
\eqref{eq:power-erasure-conclusion}.
\end{proof}

\begin{corollary}[Homogenization at small squared Hellinger distance]
\label{cor:small-defect-H}
Put
\begin{equation}
 \zeta_c:=\frac1{2(a+1)}.
 \label{eq:zeta-c}
\end{equation}
If $\sum_i d(\mu_i)\le1/4$, then
\begin{equation}
 \Df(\overline\mu^{\star n})
 \le
 C_{\mathrm{SD}}(c)\,
 \Df(\mu_1\star\cdots\star\mu_n),
 \label{eq:small-defect-H}
\end{equation}
where one may take
\begin{align}
 C_{\mathrm{SD}}(c)
 &:=84C_{\mathrm{end}}(c)c^2
 \left(1-2^{-\zeta_c}\right)^{-(a+1)}
 \label{eq:CSD-explicit}\\
 &\le
 504e^{1/4}\cosh(1/4)c^5
 \left(\frac{4(a+1)}{\log2}\right)^{a+1}.
 \label{eq:CSD-coarse}
\end{align}
\end{corollary}

\begin{proof}
Apply Corollary~\ref{cor:likelihood-erasure} and
Theorem~\ref{thm:power-erasure} with
\[
 L=42C_{\mathrm{end}}(c),
 \qquad
 \beta=\frac12,
 \qquad
 \gamma=\zeta_c.
\]
Then $\beta-a\gamma=\gamma$. Moreover,
$q_k\le1/k!\le2^{-(k-1)}$, so
\[
 S_{\zeta_c}
 \le\sum_{k\ge1}2^{-\zeta_c(k-1)}
 =\frac1{1-2^{-\zeta_c}}.
\]
Equation \eqref{eq:CSD-explicit} follows. Since
$1-e^{-x}\ge x/2$ for $0\le x\le1$,
\[
 \frac1{1-2^{-\zeta_c}}
 \le\frac{2}{\zeta_c\log2}
 =\frac{4(a+1)}{\log2}.
\]
Finally, $2c^2+1\le3c^2$ in
\eqref{eq:Cend-explicit}, proving \eqref{eq:CSD-coarse}.
\end{proof}

\section{Completion and approximation}\label{app:completion}

\begin{proof}[Completion of the proof of Theorem~\ref{thm:analytic}]
By Proposition~\ref{prop:approximation}, it suffices first to treat
positive finitely supported mean-one laws. Put
\[
 \Delta:=\sum_{i=1}^n d(\mu_i).
\]
If $\Delta\le1/4$, Corollary~\ref{cor:small-defect-H} gives
\[
 \Df(\overline\mu^{\star n})
 \le
 C_{\mathrm{SD}}(c)\,
 \Df(\mu_1\star\cdots\star\mu_n).
\]
If $\Delta>1/4$, Theorem~\ref{thm:large-defect}, with $d_0=1/4$,
gives
\[
 \Df(\overline\mu^{\star n})
 \le
 C_{\mathrm{LD}}(c,1/4)\,
 \Df(\mu_1\star\cdots\star\mu_n).
\]
Thus \eqref{eq:H} holds for positive finitely supported laws with
\begin{equation}
 \Cuniv_0(K)
 :=
 \max\left\{
  C_{\mathrm{SD}}(c),
  R_{\alpha_c}\bigl[1+8c^3(4D_c^4+2)\bigr]
 \right\},
 \label{eq:universal-modulus-preliminary}
\end{equation}
where
\[
 c=\max\{2,4K-2\},
 \qquad
 D_c=6c+2,
 \qquad
 \alpha_c=2\log_2D_c.
\]
The least admissible $R_\alpha$ in
Theorem~\ref{thm:exp-regular-cover} is nondecreasing in $\alpha$, since
the class of functions satisfying \eqref{eq:Psi-moderate} increases
with $\alpha$. Hence $\Cuniv_0$ may be chosen nondecreasing. Define its
right-continuous envelope by
\begin{equation}
 \Cuniv(K):=\inf_{\varepsilon>0}\Cuniv_0(K+\varepsilon).
 \label{eq:universal-modulus}
\end{equation}
The function $\Cuniv$ is universal, nondecreasing, right-continuous, and
dominates $\Cuniv_0$. The preceding estimates therefore hold with
$\Cuniv$, independently of $f$ and $n$. If the two-coordinate hypothesis
holds with every constant $K+\varepsilon$, apply the estimates at
$K+\varepsilon$ and take the infimum over $\varepsilon>0$.
Proposition~\ref{prop:approximation} extends the estimate, with the same
constant, to arbitrary mean-one laws on $[0,\infty)$, including atoms at
zero and extended-valued divergences.
\end{proof}


\subsection{Approximation and endpoints}\label{app:approximation}

\begin{proposition}[Approximation principle]
\label{prop:approximation}
Suppose that \eqref{eq:H}, with one fixed constant, is known for positive
finitely supported mean-one laws. Then it holds with the same constant
for arbitrary mean-one laws on $[0,\infty)$, including atoms at zero and
extended-valued divergences.
\end{proposition}

\begin{proof}
First suppose that each $\mu_i$ is supported on $(0,\infty)$. Let
$X_i\sim\mu_i$. Choose positive finite-valued simple functions $S_{i,r}$
with $\E|S_{i,r}-X_i|\to0$, and put
\[
 \mathcal G_{i,r}:=\sigma(S_{i,1},\ldots,S_{i,r}),
 \qquad
 X_{i,r}:=\E[X_i\mid\mathcal G_{i,r}],
 \qquad
 \mu_{i,r}:=\Law(X_{i,r}).
\]
The sigma-fields $\mathcal G_{i,r}$ are finite and increasing. Since
$S_{i,r}$ is $\mathcal G_{i,r}$-measurable, conditional Jensen gives
\[
 \E|X_i-X_{i,r}|
 \le
 \E|X_i-S_{i,r}|
 +\E\left|\E[S_{i,r}-X_i\mid\mathcal G_{i,r}]\right|
 \le
 2\E|X_i-S_{i,r}|.
\]
Thus $X_{i,r}\to X_i$ in $L^1$. Passing to a common subsequence, still
indexed by $r$, choose it so that
\[
 \sum_{r=1}^\infty
 2^r\sum_{i=1}^n\E|X_{i,r}-X_i|<\infty.
\]
Markov's inequality and the union bound then give
\[
 \sum_{r=1}^\infty
 \Pp\left\{\max_i|X_{i,r}-X_i|>2^{-r}\right\}<\infty.
\]
By Borel--Cantelli, $X_{i,r}\to X_i$ almost surely for every $i$.
Each $X_{i,r}$ is positive, finitely supported, and mean one.

Take the rows independently. On every atom of
$\bigvee_i\mathcal G_{i,r}$, independence gives
\[
 \prod_{i=1}^nX_{i,r}
 =
 \E\!\left[
  \prod_{i=1}^nX_i
  \;\middle|\;
  \bigvee_{i=1}^n\mathcal G_{i,r}
 \right].
\]
Conditional Jensen therefore gives
\begin{equation}
 \Df(\mu_{1,r}\star\cdots\star\mu_{n,r})
 \le
 \Df(\mu_1\star\cdots\star\mu_n).\label{eq:sep-approx}
\end{equation}

Let $I_1,\ldots,I_n$ be iid uniform on $[n]$. For each $i,j$, take an
independent copy of the entire coupled sequence
$(X_i,(X_{i,r})_{r\ge1})$, denoted
$(X_{i,j},(X_{i,j,r})_{r\ge1})$, and put
\[
 Y_j:=X_{I_j,j},
 \qquad
 Y_{j,r}:=X_{I_j,j,r}.
\]
The $Y_j$ are iid with common law
$\overline\mu=n^{-1}\sum_i\mu_i$, while the $Y_{j,r}$ are iid with common
law $\overline\mu_r=n^{-1}\sum_i\mu_{i,r}$. The chosen coupling gives
$Y_{j,r}\to Y_j$ almost surely for every $j$. Lower semicontinuity of
$f$ and Fatou's lemma imply
\begin{equation}
 \Df(\overline\mu^{\star n})
 \le
 \liminf_{r\to\infty}\Df(\overline\mu_r^{\star n}).\label{eq:hom-approx}
\end{equation}
The finite-support estimate, \eqref{eq:sep-approx}, and
\eqref{eq:hom-approx} prove the assertion for positive laws.

For laws with atoms at zero, fix $\varepsilon\in(0,1)$ and set
\[
 X_i^{(\varepsilon)}:=\varepsilon+(1-\varepsilon)X_i.
\]
Let $B_i\sim\Ber(1-\varepsilon)$ independently and put
$W_i=B_iX_i+(1-B_i)$. Then
$X_i^{(\varepsilon)}=\E[W_i\mid X_i]$, and independence of the $B_i$
gives
\[
 \prod_iX_i^{(\varepsilon)}
 =\E\!\left[\prod_iW_i\mid X_1,\ldots,X_n\right].
\]
Conditional Jensen therefore yields
\[
 \Df\!\left(\Astar_i\Law(X_i^{(\varepsilon)})\right)
 \le
 \Df\!\left(\Astar_i\Law(W_i)\right).
\]
Conditioning on the Bernoulli pattern, the product of the $W_i$ is a
product over a subset of the original rows. Every such subset divergence
is at most the full product divergence by Lemma~\ref{lem:monotonicity}.
Thus
\begin{equation}
 \Df\!\left(\Astar_i\Law(X_i^{(\varepsilon)})\right)
 \le
 \Df(\mu_1\star\cdots\star\mu_n).\label{eq:zero-reg-sep}
\end{equation}
Apply the positive-law theorem to the regularized laws, let
$\varepsilon\downarrow0$, and use Fatou on the homogenized side.
\end{proof}


\appendix

\section{Equivalent criteria and growth bounds}
\label{app:scalar}

We record consequences and reformulations of
\eqref{eq:endpoint-test} that are independent of the proof of
Theorem~\ref{thm:analytic}.

\subsection{Separate growth conditions}

For $s>0$, define the secant slopes
\begin{equation}
 u(s)=\frac{f(e^s)}{e^s-1},
 \qquad
 v(s)=\frac{f(e^{-s})}{1-e^{-s}}.
 \label{eq:local-secants}
\end{equation}
Both functions are nonnegative and nondecreasing by convexity.

\begin{proposition}
\label{prop:local-separate-growth}
If $u(2s)\le C u(s)$ and $v(2s)\le C v(s)$ for all $s>0$, then
\eqref{eq:endpoint-test} holds with $M=2C$.
Conversely, if the endpoint criterion holds, there are finite constants
$C,p_+,p_-$, depending on $f$, such that
\begin{align}
 f(t)&\le Ct(1+\log t)^{p_+},&&t\ge2,
 \label{eq:local-upper-growth}\\
 f(t)&\le C(1+|\log t|)^{p_-},&&0<t\le\tfrac12.
 \label{eq:local-lower-growth}
\end{align}
\end{proposition}

\begin{proof}
Writing $r=e^{-x}$ and $t=e^y$, direct substitution in
\eqref{eq:local-two-point-law} gives
\[
 \Df(\rho_{r,t})
 =\frac{(t-1)(1-r)}{t-r}\,[v(x)+u(y)].
\]
The endpoint part of the square, divided by this common prefactor, is
\begin{equation}
 \frac{(t-1)(1+r)}{t-r}v(2x)
 +\frac{(1-r)(t+1)}{t-r}u(2y).
 \label{eq:local-secant-endpoints}
\end{equation}
The two coefficients are nonnegative and sum to $2$. Separate doubling
therefore bounds \eqref{eq:local-secant-endpoints} by
$2C[v(x)+u(y)]$, proving sufficiency.

For necessity, fix $r\in(0,1)$. The coefficient of $u(2y)$ in
\eqref{eq:local-secant-endpoints} is at least $1-r$, so the endpoint
criterion gives
\[
 u(2y)\le\frac{M}{1-r}[u(y)+v(-\log r)].
\]
After adding a fixed positive constant to $u$, dyadic iteration bounds
it by a polynomial in $1+y$. Fixing $t>1$ gives the same conclusion
for $v(x)$. Substitution in \eqref{eq:local-secants} proves
\eqref{eq:local-upper-growth}--\eqref{eq:local-lower-growth}.
\end{proof}

For differentiable $f$ with $f(t)>0$ for $t\ne1$, a sufficient condition is
\[
 \sup_{s>0}\max\left\{
 \frac{s u'(s)}{u(s)},\frac{s v'(s)}{v(s)}\right\}\le p<\infty.
\]
Integration from $s$ to $2s$ gives separate doubling with constant
$2^p$. In terms of $f$, the two expressions are
\[
 \log t\left[\frac{tf'(t)}{f(t)}-\frac{t}{t-1}\right],\quad t>1,
 \qquad
 (-\log t)\left[-\frac{tf'(t)}{f(t)}-\frac{t}{1-t}\right],\quad 0<t<1.
\]
In the symmetric case, if $f$ is positive and continuously
differentiable on $(1,\infty)$, the necessary and sufficient condition
\eqref{eq:local-symmetric-doubling} can also be written
\[
 \sup_{t>1}\left\{
 \int_t^{t^2}\frac{f'(z)}{f(z)}\,\mathrm dz-\log(t+1)\right\}<\infty.
\]

\subsection{Posterior coordinates}

On $(-1,1)$, set
\[
 \Phi(u)=(1-u)f\!\left(\frac{1+u}{1-u}\right),
 \qquad T(u)=\frac{2u}{1+u^2}.
\]
The function $\Phi$ is convex, and squaring the likelihood ratio sends
$u=(t-1)/(t+1)$ to $T(u)$. If
$r=(1-a)/(1+a)$ and $t=(1+b)/(1-b)$, the endpoint criterion is exactly
\begin{equation}
\begin{split}
 &b^2(1+a^2)\Phi(-T(a))+a^2(1+b^2)\Phi(T(b))\\
 &\hspace{5mm}\le
 M(a+b)\bigl[b\Phi(-a)+a\Phi(b)\bigr],\qquad 0<a,b<1.
\end{split}
\label{eq:local-posterior-criterion}
\end{equation}
Indeed, under the midpoint measure $(P+Q)/2$, the two posterior values
$-a,b$ have probabilities $b/(a+b),a/(a+b)$. Thus the original
divergence is $[b\Phi(-a)+a\Phi(b)]/(a+b)$, and the two repeated
outcomes contribute the left-hand side of
\eqref{eq:local-posterior-criterion} divided by $(a+b)^2$.

\section*{Statements and declarations}

Generative-AI tools were used during exploratory development and editorial preparation of this draft. The author takes responsibility for all mathematical statements and references.



\end{document}